\documentclass[a4paper,reqno]{amsart}
\usepackage{newcent} 
\usepackage[T1]{fontenc}
\usepackage[centertags]{amsmath}
\usepackage{amsfonts}
\usepackage{amssymb}
\usepackage{amsthm}
\usepackage{newlfont}
\usepackage{fancyhdr,fancyvrb}
\usepackage[dvips]{graphicx,subfig}
\usepackage{pinlabel}
\usepackage{nextpage}
\usepackage{lscape}

\newcommand{\R}[1]{\mathbb{R}^{#1}}

\newcommand{\cP}{\mathcal P}
\newcommand{\cR}{\mathcal R}

\newcommand{\ds}{\de s}
\newcommand{\dsigma}{\de\sigma}

\newcommand{\rc}{\mathrm c}
\newcommand{\de}{\mathrm d}

\newcommand{\fin}{\mathrm{fin}}

\newcommand{\init}{\mathrm{in}}

\newcommand{\rig}{\mathrm{right}}
\newcommand{\sh}{\mathrm{sh}}

\newcommand{\cl}[1]{\overline{#1}}

\newcommand{\eps}{\varepsilon}

\newcommand{\norm}[1]{\left\lvert#1\right\rvert}
\newcommand{\norma}[1]{\left\lVert#1\right\rVert}

\newcommand{\scp}[2]{\left\langle#1,#2\right\rangle}

\newcommand{\Cl}[2]{C^{#1}_{#2}}

\newcommand{\Hp}[2]{H^{#1}_{#2}}
\newcommand{\Lp}[2]{L^{#1}_{#2}}

\newcommand{\spt}{\operatorname{spt}}

\newcommand{\sign}{\operatorname{sign}}

\DeclareMathOperator*{\supess}{ess\,sup}

\renewcommand{\geq}{\geqslant}
\renewcommand{\leq}{\leqslant}
\newcommand{\wto}{\rightharpoonup}

\newcommand{\be}{\begin{equation}}
\newcommand{\ee}{\end{equation}}

\renewcommand{\vec}[2]{\left(\begin{array}{c}
                    #1 \\
                    #2
                    \end{array}\right)}
\newcommand{\mat}[4]{\left[\begin{array}{cc}
                    #1 & #2 \\
                    #3 & #4
                    \end{array}\right]}

\mathchardef\emptyset="001F

\numberwithin{equation}{section}

\newtheorem{defin}{Definition}[section]
\newtheorem{theorem}[defin]{Theorem}
\newtheorem{lemma}[defin]{Lemma}
\newtheorem{proposition}[defin]{Proposition}

\theoremstyle{definition} {\newtheorem{remark}[defin]{Remark}}

\title[One-dimensional swimmers: dynamics, controllability, and optimal controls]{One-dimensional swimmers in viscous fluids:\\ dynamics, controllability, and existence\\ of optimal controls}%
\author{Gianni Dal Maso}
\address{SISSA\\ Via Bonomea, 265\\ I--34136 Trieste, Italy}
\email[G.~Dal Maso]{dalmaso@sissa.it}
\author{Antonio DeSimone}
\address{SISSA\\ Via Bonomea, 265\\ I--34136 Trieste, Italy}
\email[A.~DeSimone]{desimone@sissa.it}
\author{Marco Morandotti}
\address{Departamento de Matem\'atica \\ Instituto Superior T\'ecnico \\ Av.\@ Rovisco Pais, 1 \\ 1049-001 Lisboa, Portugal}
\email[M.~Morandotti]{marco.morandotti@tecnico.ulisboa.pt}

\date{19 April 2014}

\makeindex%

\begin{document}

\begin{abstract}
In this paper we study a mathematical model of  one-dimensional swimmers performing a planar motion while fully immersed in a viscous fluid. 
The swimmers are assumed to be of small size, and all inertial effects are neglected.
Hydrodynamic interactions are treated in a simplified way, using the local drag approximation of resistive force theory. We prove existence and uniqueness of the solution of the equations of motion driven by shape changes of the swimmer.
Moreover, we prove a controllability result showing that given any pair of initial and final states, there exists a history of shape changes such that the resulting motion takes the swimmer from the initial to the final state. 
We give a constructive proof, based on the composition of elementary maneuvers (straightening and its inverse, rotation, translation), each of which represents the solution of an interesting motion planning problem.
Finally, we prove the existence of solutions for the optimal control problem of finding, among the histories of shape changes taking the swimmer from an initial to a final state, the one of minimal energetic cost.
\end{abstract}
\maketitle%
{\small
\keywords{\noindent {\bf Keywords:} {motion in viscous fluids, fluid-solid interaction, micro-swimmers, resistive force theory, controllability, optimal control.}
}

\bigskip
\subjclass{\noindent {\bf {2010}
Mathematics Subject Classification:}
{
Primary 76Z10;	
Secondary 74F10,  	
 49J21,  	
 93B05.  	
 }}
}\bigskip
\bigskip

\tableofcontents%

\section{Introduction}

\par In this paper we study the self-propelled planar motions of a one-dimensional swimmer in an infinite viscous three-dimensional fluid. 
We are interested in the swimming strategies of small organisms that achieve self-propulsion by propagating bending waves along their slender bodies (such as, for instance, sperm cells and \emph{Caenorhabditis elegans}). At these length scales, viscosity dominates over inertia: accordingly, we ignore all inertial effects in our analysis.

The study of the self-propulsion strategies of microscopic living organisms is attracting increasing attention, starting from  seminal works by 
Taylor \cite{Taylor1951}, 
Lighthill \cite{Lighthill1952}, 
Purcell \cite{Purcell77}, and 
Childress \cite{Childress1981}. We refer the reader to the recent review \cite{LP09} for a comprehensive list of references. Among the recent mathematical contributions we quote \cite{KoillerEhlersMontgomery1996, Galdi1999, SanMartinTakahashiTucsnak2007,Bressan2008,ADSL2008,ChambrionMunnier2011,AHMDS2012}. 
Many of these papers approach swimming problems within the framework of control theory, and this is exploited in 
\cite{ADSH2011,ADSL2009} for the numerical computation of energetically optimal strokes.
While the connection between swimming and control theory is very natural, only recently has this point of view started to emerge and become widely appreciated, see \cite{DSHAL12} and the other chapters in the same volume.

\par When inertial forces are neglected, and external forces such as gravity are not present (neutrally buoyant swimmers), 
the equations of motion for a swimmer become the statements that the total viscous force and torque exerted by the surrounding fluid vanish. 
In order to take advantage of the simplifications deriving from the special one-dimensional geometry of our swimmers, 
we adopt here the local drag approximation of   \emph{Resistive Force Theory}, first introduced in \cite{GH1955}, then also used in \cite{PK74}, and further discussed in \cite{JohnsonBrokaw1979}.
It is a classical and popular theory widely spread among biological fluid dynamicists, 
which has recently been proved to be accurate and robust in the study of the motion of one-dimensional bodies in the length scales and regimes we are interested in, 
as it is shown, e.g., in \cite{FRKHJ2010}.
According to resistive force theory,  the external fluid exerts on the swimmer a viscous force per unit length which, at each point of the swimmer, 
is proportional to the local tangential and normal velocities \emph{at that point}, 
through positive resistance coefficients denoted by $C_\tau$ and $C_\nu$, respectively.

\par For every $t$ in the time interval $[0,T]$, let $s\mapsto\chi(s,t)$ be the parametrization of the swimmer position 
with respect to an absolute external reference frame (\emph{lab frame}), where $s\in[0,L]$ is the arc length parameter. 
It is possible to factorize this function as $\chi(s,t)=r(t)\circ\xi(s,t)$, 
where $r(t)$ is a time dependent rigid motion and $s\mapsto\xi(s,t)$ describes the shape of the swimmer at time $t$ 
with respect to a reference system moving with the swimmer (\emph{body frame}).

\par We suppose that the shape function $\xi$ is given. 
The first problem we address in this paper is to determine the rigid motion $t\mapsto r(t)$
that results from a prescribed time history of shape changes $t\mapsto\xi(s,t)$.  
This is obtained by imposing that $\chi=r\circ\xi$ satisfies the equations of motion (the resultant of viscous forces and torques generated by the interaction between the swimmer and the fluid vanish for every $t$) and solving the resulting force and torque balance for $r$ in terms of the given $\xi$.

\par Our main result on this first problem is that, if $\xi$ satisfies suitable regularity conditions which are listed in the hypotheses of Theorem \ref{0001}, 
then the rigid motion $r(t)$ can be determined as the unique solution of  a system of ordinary differential equations in the independent variable $t$. 
Therefore, for every initial condition $r_0$, there exists a unique $r(t)$ such that the resulting function $\chi(s,t)=r(t)\circ\xi(s,t)$ satisfies the force and torque balance.
In other words, Theorem \ref{0001} states that looking for a motion that satisfies the force and torque balance is equivalent to assigning the shape function and solving the equations of motion.

\par The second problem we address in this paper is that of controllability. Given a time interval $[0,T]$ and 
arbitrary initial and final states of the swimmer described by
the arc length parametrizations $s\mapsto\chi_\init(s)$ and $s\mapsto\chi_\fin(s)$, can we find a self-propelled motion $\chi(s,t)$ in the lab frame such that $\chi(s,0)=\chi_\init(s)$ and $\chi(s,T)=\chi_\fin(s)$\,?
By a self-propelled motion we mean one such that the equations of motion are satisfied which, in the case of self-propulsion, reduce to the vanishing of the total viscous force and torque. 
The answer is affirmative and is contained in Theorem \ref{300}. 
Our proof is constructive. 
Indeed, we exhibit an explicit procedure to transfer $\chi_\init$ onto $\chi_\fin$ based on the composition of elementary maneuvers: straightening of a curved configuration and the corresponding inverse maneuver (i.e., how to map a straight segment onto a given curved configuration), rotation of a straight segment around its barycenter, translation of a segment along its axis. 
Solving the motion planning problem for these elementary maneuvers is interesting in its own right, independently of the general controllability result, and this is done in Section \ref{controllability}.

\par More in detail, given two configurations $\chi_\init$ in $\chi_\fin$, we show how  to straighten them in a segment-like configuration, say $\Sigma_\init$ and $\Sigma_\fin$, respectively, thanks to Theorem \ref{0001}.
Then we show how to transfer $\Sigma_\init$ into $\Sigma_\fin$, by explicitly constructing a way to make a rectilinear swimmer to translate (without rotating) along its axis, see Section \ref{translation}, and a way to make it rotate (without translating) about its barycenter,  see Section \ref{rotation}.
These constructions use suitable bending wave forms that propagate along the body of the swimmer.

\par It is interesting to notice, and this will be clear in Section \ref{controllability}, 
that a very convenient way to describe such transformations is by using the angle that the tangent of the swimmer makes with the positive horizontal axis. 
This angle is given as a function of the time $t$ and of the arc length parameter $s$.
This agrees with the traditional approach of prescribing the curvature function, 
since the latter can be recovered by differentiating the angle with respect to~$s$ (see Remark \ref{curv}). 
This classical approach is motivated by the fact that the swimmers we are interested in accomplish the shape changes required for force and torque balance by relative sliding of filaments along their ``spine'', hence inducing local curvature changes.

\par The last problem we address is the existence of an energetically optimal swimming strategy. 
Here again we rely crucially on the simplification yielded by Resistive Force Theory since obtaining a similar result when the fluid-swimmer interaction is modeled by the Stokes equations is much more involved.
In Theorem \ref{010} we prove that, under suitable conditions, there exists a self-propelled motion $\chi(s,t)$ minimizing the power expended. 
The key hypothesis is a  sort of non-interpenetra\-tion condition for the enlarged body obtained by thickening the curve describing the swimmer to a tube of constant thickness. 
This condition rules out self-intersections of the swimmer and yields an a-priori bound on its curvature.

\section{Mathematical statement of the problem}

\par In this section we describe the mathematical setting for studying the swimming problem by adapting  to our specific case of a one-dimensional swimmer with a local fluid-swimmer interaction the framework introduced and described in \cite{DMDSM11,M11}.
\par Throughout the paper we fix $L>0$ to be the length of the swimmer and $T>0$ so that $[0,T]$ is the time interval in which the motion occurs. We study planar motions in three dimensions, and therefore the position of each material point of the swimmer will be described by a function $\chi\colon[0,L]{\times}[0,T]\to\R2$, where $s\in[0,L]$ is the arc length parameter; this request means that for every $t$ the map $s\mapsto\chi(s,t)$ is Lipschitz continuous from $[0,L]$ to $\R2$ and $\norm{\chi'(s,t)}\equiv1$, where $\chi':=\partial\chi/\partial s$. As for the derivative with respect to $t$, $\dot\chi:=\partial\chi/\partial t$ is intended in the distributional sense as the object that makes the following equality hold true
$$\int_0^L\!\!\int_0^T \dot\chi(s,t)\varphi(s,t)\,\de s\de t=-\int_0^L\!\!\int_0^T \chi(s,t)\frac{\partial\varphi(s,t)}{\partial t}\,\de s\de t\,,$$
for every $\varphi\in\Cl\infty\rc((0,L){\times}(0,T))$.

\par We now introduce the local expressions for the line densities $f(s,t)$ and $m(s,t)$ of viscous force and torque, as dictated by resistive force theory. Since $f(s,t)$ lies in the plane of the motion, $m(s,t)$ is orthogonal to it and is identified with a scalar. They are given by
\begin{equation}\label{104}
\begin{split}
f(s,t)=&\;-[C_\tau\, \dot\chi_\tau(s,t)\,\chi'(s,t)+C_\nu\, \dot\chi_\nu(s,t)\,J\chi'(s,t)]=-K_\chi(s,t)\dot\chi(s,t)
, \\
m(s,t) =&\;\scp{f(s,t)}{J\chi(s,t)}=-\scp{C_\tau\, \dot\chi_\tau(s,t)\,\chi'(s,t)+C_\nu\, \dot\chi_\nu(s,t)\, J\chi'(s,t)}{J\chi(s,t)}\\
=&-\scp{K_\chi(s,t)\dot\chi(s,t)}{J\chi(s,t)}.
\end{split}
\end{equation}
Here, $C_\tau$ and $C_\nu$ are positive constants, $\dot\chi_\tau$ and $\dot\chi_\nu$ are the tangential and normal components of the velocity $\dot\chi$, i.e., $\dot\chi_\tau(s,t)=\langle\dot\chi(s,t),$ $\chi'(s,t)\rangle$ and $\dot\chi_\nu(s,t)=\scp{\dot\chi(s,t)}{J\chi'(s,t)}$, while $J=\mat0{-1}10$ is the counter-clockwise rotation matrix of angle $\pi/2$ and
\begin{equation}\label{0000}
K_\chi(s,t):=C_\tau\chi'(s,t)\otimes\chi'(s,t)+C_\nu(J\chi'(s,t))\otimes(J\chi'(s,t)),
\end{equation}
where for any two vectors $a,b\in\R2$ the matrix $a\otimes b$ is defined by $(a\otimes b)_{ij}=a_ib_j$\,.
The force and torque balance can be written as
\begin{subequations}\label{105}
\begin{eqnarray}
0=F(t) &\!\!\!\!:=&\!\!\!\! \int_0^L f(s,t)\,\de s=-\int_0^L K_\chi(s,t)\dot\chi(s,t)\,\de s, \label{105a}\\
0=M(t) &\!\!\!\!:=&\!\!\!\! \int_0^L m(s,t)\,\de s=-\int_0^L \scp{K_\chi(s,t)\dot\chi(s,t)}{J\chi(s,t)}\de s, \label{105b}
\end{eqnarray}
\end{subequations}
for a.e. $t\in[0,T]$.

\begin{remark}\label{rem}
An important remark on the structure of the viscous force and torque is in order, leading to a rate independence property. 
Let $\varphi$ be a $\Cl1{}$ strictly increasing function with $\varphi'(t)>0$ for every $t\in[0,T]$. 
Then a rescaling in time by $\varphi$ has no consequences on the equations of motion. 
Indeed, we prove that if $\chi(s,t)$ satisfies the force and torque balance \eqref{105}, then also $(\chi\circ\varphi)(s,t):=\chi(s,\varphi(t))$ does. 
Let us rewrite the force \eqref{105a} as $F_\chi(t)=-\int_0^L K_\chi(s,t)\dot\chi(s,t)\,\de s$. 
Then we have
\begin{equation*}
\begin{split}
F_{\chi\circ\varphi}(t)= & -\int_0^L K_{\chi\circ\varphi}(s,\varphi(t))\dot\chi(s,\varphi(t))\dot\varphi(t)\,\de s\\
=& -\dot\varphi(t)\int_0^L K_\chi(s,\varphi(t))\dot\chi(s,\varphi(t))\,\de s=\dot\varphi(t)F_\chi(\varphi(t))=0,
\end{split}
\end{equation*}
since $F_\chi(\varphi(t))=0$ for a.e.\ $t\in[0,T]$. The same can be obtained for the torque $M_{\chi\circ\varphi}(t)$.

\par This rate independece character is at the root of the celebrated Scallop Theorem, see \cite{Purcell77}.
\end{remark}

\par We conclude this section by introducing a function space $X$ containing our state functions, as well as the shape functions:
\begin{equation}\label{308}
X:=\{\chi\colon [0,L]{\times}[0,T]\to\R2:  \chi\in\Lp\infty{}(0,T;\Hp2{}(0,L)),\,
\dot\chi\in\Lp2{}(0,T;\Lp2{}(0,L))\},
\end{equation}
endowed with the norm
\begin{equation}\label{310}
\norma\chi_X:=\supess_{0\leq t\leq T}\norma{\chi(\cdot,t)}_{\Hp2{}(0,L)}+\left(\int_0^T \lVert\dot\chi(\cdot,t)\rVert_{\Lp2{}(0,L)}^2\,\de t\right)^{1/2},
\end{equation}
which makes it a Banach space. It follows from the definition that
\begin{equation}\label{3101}
X\subset \Hp1{}(0,T;\Lp2{}(0,L)) \text{ with continuous embedding}.
\end{equation}
Since every function $\chi$ in $\Hp1{}(0,T;\Lp2{}(0,L))$ can be modified on a negligible subset of $[0,T]$ so that $t\mapsto \chi(\cdot,t)$ is strongly continuous from $[0,T]$ into $\Lp2{}(0,L)$, we shall always refer to this modified function when we consider the properties of $\chi(\cdot,t)$ for some $t\in[0,T]$. With this convention we have
\begin{equation}\label{980}
X\subset \Cl0{}([0,T];\Lp2{}(0,L)) \text{ with continuous embedding}.
\end{equation}
The following proposition shows the main properties of the space $X$.
\begin{proposition}
Let $\chi\in X$. Then for every $t\in[0,T]$ we have $\chi(\cdot,t)\in\Hp2{}(0,L)$ and
\begin{equation}\label{833}
\norma{\chi(\cdot,t)}_{\Hp2{}(0,L)}\leq\norma\chi_X.
\end{equation}
Moreover, the function $t\mapsto\chi(\cdot,t)$ is continuous with respect to the weak topology of $\Hp2{}(0,L)$. Finally,
\begin{eqnarray}
\label{3103}
&\chi\in \Cl0{}([0,T];\Cl1{}([0,L])),
\\
\label{834}
&\norma\chi_{\Cl0{}([0,T];\Cl1{}([0,L]))}\leq C\norma\chi_X\,,
\end{eqnarray}where the constant $C$ is independent of $\chi$.
\end{proposition}
\begin{proof}
To prove the first claim, let us fix $t\in[0,T]$ and let $N$ be a zero measure set up to which the essential supremum in \eqref{310} is actually a supremum. Consider a sequence $t_n\notin N$ converging to $t$, so that $\norma{\chi(\cdot,t_n)}_{\Hp2{}(0,L)}\leq\norma\chi_X$\,.
Since $\chi(\cdot,t_n)\to\chi(\cdot,t)$ in $L^2([0,L])$ by \eqref{980}, we have that $\chi(\cdot,t)\in\Hp2{}(0,L)$. Moreover, since the $\Hp2{}$ norm is lower-semicontinuous, we have also
$$
\norma{\chi(\cdot,t)}_{\Hp2{}(0,L)}\leq\liminf_{n\to+\infty}\norma{\chi(\cdot,t_n)}_{\Hp2{}(0,L)}\leq\norma\chi_X,
$$
which proves \eqref{833}. Thanks to this inequality, the strong continuity of the function $t\mapsto\chi(\cdot,t)$ in $\Lp2{}([0,L])$ implies the weak continuity in $\Hp2{}(0,L)$. Then the compact embedding of $H^2(0,L)$ into $C^1([0,L])$ implies that the function $t\mapsto\chi(\cdot,t)$ is continuous with respect to the strong topology of $C^1([0,L])$, which gives \eqref{3103}.
Finally, \eqref{834} follows from \eqref{833} and from the continuous embedding of $\Hp2{}(0,L)$ into $\Cl1{}([0,L])$.
\end{proof}

Note that, by \eqref{3103}, for every $\chi\in X$ we have
\begin{equation}\label{8360}
\chi,\,\chi'\in \Cl0{}([0,L]{\times}[0,T]).
\end{equation}
We are interested only in functions $\chi\in X$ such that $s\mapsto\chi(s,t)$ is the arc length parametrization of a curve; this leads to the following definition
\begin{equation}\label{836}
X_1:=\{\chi\in X:\norm{\chi'}=1\text{ in $[0,L]{\times}[0,T]$}\}.
\end{equation}

\section{Equations of motion}
\par In this section we derive the equations of motion for the swimmer. 
It is convenient to factorize the function $\chi\in X_1$ as the composition of a time dependent rigid motion $r$, which represents the change of location, with a function  $\xi\in X_1$, which represents the change of shape.
We write
\begin{equation}\label{1001}
\chi(s,t)=x(t)+R(t)\xi(s,t),
\end{equation}
where $x(t)\in \R2$ is the translation vector and $R(t)\in \R{2{\times}2}$ is the rotation corresponding to the rigid motion $r(t)$. 

\par If we assume that 
$\int_0^L \xi(s,t)\,\de s=0$ for every $t\in[0.T]$, then $x(t)$ coincides with the barycenter of the curve $\chi(\cdot,t)$, which describes the swimmer at time $t$ with respect to the absolute reference system, while
the function $\xi(\cdot,t)$ will be regarded as the deformation seen by an observer moving with barycenter of the swimmer. 

\begin{proposition}\label{752}
Let $\xi\in X_1$ and let $x\colon [0,T]\to\R2$ and $R\colon [0,T]\to\R{2{\times}2}$ be functions such that $R(t)$ is a rotation for every $t\in[0,T]$. Then the following properties are equivalent:
\begin{itemize}
\item[(i)] the function $\chi$ defined by \eqref{1001} belongs to $X_1$\,;
\item[(ii)] the functions $x$ and $R$ belong to $\Hp1{}(0,T)$.
\end{itemize}
\end{proposition}
\begin{proof}
(i)$\Rightarrow$(ii).
For every $t\in[0,T]$ we define
$$
\overline\xi(t):=\frac1L\int_0^L \xi(s,t)\,\de s \qquad\text{and}\qquad
\overline\chi(t):=\frac1L\int_0^L \chi(s,t)\,\de s.
$$
Since $\xi,\,\chi\in X_1\subset \Hp1{}(0,T;\Lp2{}(0,L))$ we have that $\overline\xi,\,\overline\chi \in\Hp1{}(0,T)$.
By averaging \eqref{1001} with respect to $s$ we obtain
\begin{equation}\label{837}
\overline\chi(t)=x(t)+ R(t)\overline\xi(t).
\end{equation}
Subtracting this equation from \eqref{1001} we obtain
\begin{equation}\label{839}
\chi(s,t)-\overline\chi(t)=R(t)\big(\xi(s,t)-\overline\xi(t)\big).
\end{equation}

\par Let us fix $t_0\in[0,T]$. Since $|\xi'(s,t_0)|=1$ for every $s\in [0,L]$, there exists $s_0\in [0,L]$ such that $\xi(s_0,t_0)-\overline\xi(t_0)\neq 0$.
By the continuity of $\xi$ there exist $\eps>0$, an open neighborhood $U$ of $s_0\in[0,L]$, and an open neighborhood $V$ of $t_0$ in $[0,T]$ such that $\norm{\chi(s,t)-\overline\chi(t)}=\norm{\xi(s,t)-\overline\xi(t)}\geq\eps$ for every $s\in U$ and every $t\in V$, where the equality follows from \eqref{839}. Let $\chi^*(s,t):=\big(\chi(s,t)-\overline\chi(t)\big)/\norm{\chi(s,t)-\overline\chi(t)}$ and $\xi^*(s,t):=\big(\xi(s,t)-\overline\xi(t)\big)/\norm{\xi(s,t)-\overline\xi(t)}$. By \eqref{839} we have
\begin{equation*}
\chi^*(t,s)=R(t)\xi^*(t,s)
\end{equation*}
for every $s\in U$ and every $t\in V$. By elementary Linear Algebra we have
\begin{equation}\label{841}
R(t)=\mat{\scp{\chi^*(s,t)}{\xi^*(s,t)}}{-\scp{\chi^*(s,t)}{J\xi^*(s,t)}}{\scp{\chi^*(s,t)}{J\xi^*(s,t)}}{\scp{\chi^*(s,t)}{\xi^*(s,t)}}.
\end{equation}
By \eqref{3101} and \eqref{3103} the functions $\chi^*$ and $\xi^*$ belong to $\Hp1{}(V;\Lp2{}(U))\cap \Cl0{}(\overline V;\Cl1{}(U))$, so that the entries of the matrix in \eqref{841} belong to $\Hp1{}(V;\Lp2{}(U))$. Since the matrix does not depend on $s$, we obtain $R\in\Hp1{}(V)$. The conclusion $R\in\Hp1{}(0,T)$ follows now from a covering argument.

Since $\overline \chi$, $R$, and $\overline\xi$ belong to $\Hp1{}(0,T)$, we deduce from \eqref{837} that $x\in \Hp1{}(0,T)$.

(ii)$\Rightarrow$(i). This implication follows easily from \eqref{308} and \eqref{1001}.
\end{proof}

\begin{remark}\label{curv}
The purpose of the function $\xi(\cdot,t)$ is to describe the shape of the swimmer as a function of time. For each $t$ we can choose the most convenient reference system. Of course, different choices are compensated by different rigid motions in \eqref{1001}.

In many cases it is convenient to describe the shape of the swimmer by means of the (oriented) curvature $\kappa(s,t)$ of the curve $\xi(\cdot,t)$ at $s$. This is because both in living organisms and in technological devices shape changes are usually obtained by controlling the mutual distance of several pairs of points. Prescribing the curvature can be interpreted as the infinitesimal version of this control, whose description is easier from the mathematical point of view.

If $\chi\in X_1$ and $\xi\in X_1$ are linked by \eqref{1001}, then clearly their curvatures are the same. 
Given $\xi\in X_1$, let 
$\vartheta(s,t)$ be the oriented angle between the $x_1$-axis and the oriented tangent to the curve $\xi(\cdot,t)$ at $s$. It is well known that $\kappa(s,t)= \scp{\xi''(s,t)}{J\xi'(s,t)}=\vartheta'(s,t)$, so we can easily get $\kappa$ from $\xi$ by differentiation and $\vartheta$ from $\kappa$ by integration. In particular, if we assume $\xi'(0,t)=e_1$, we have $\vartheta(0,t)=0$, hence 
$$
\vartheta(s,t)=\int_0^s \kappa(\sigma,t)\,\de\sigma.
$$
Then the definition of $\vartheta(s,t)$ gives that $\xi'(s,t)=(\cos\vartheta(s,t),\sin\vartheta(s,t))$, so that, if $\xi(0,t)=0$, we have
$$
\xi(s,t)=\int_0^s (\cos\vartheta(\sigma,t),\sin\vartheta(\sigma,t))\,\de\sigma.
$$
This shows that the descriptions of the shape given by $\xi(s,t)$ and $\kappa(s,t)$ are equivalent.
\end{remark}

\par By the change of reference \eqref{1001}, it is possible to rephrase the force and torque balance \eqref{105} and eventually obtain ordinary differential equations governing the time evolution of $x$ and $R$. Those will be the equations of motion of the swimmer. We can write
\begin{equation*}
R(t)=\mat{\cos\theta(t)}{-\sin\theta(t)}{\sin\theta(t)}{\cos\theta(t)},
\end{equation*}
where $\theta(t)$ is the angle of rotation. We assume that $\xi$ and $\chi$ belong to $X_1$. Thanks to Proposition \ref{752}, we can differentiate \eqref{1001} with respect to time. Plugging all the terms in \eqref{105} and noticing that $K_\chi(s,t)=R(t)K_\xi(s,t)R^\top(t)$, we obtain
\begin{equation}\label{1002}
\!\!\!\!\vec{F(t)}{M_x(t)}\!=\!\mat{R(t)}001\left\{-\mat{A(t)}{b(t)}{b^\top(t)}{c(t)}\!\mat{R^\top(t)}001\!\vec{\dot x(t)}{\dot\theta(t)}+ \vec{F^\sh(t)}{M^\sh(t)}\right\},
\end{equation}
where $M_x(t):=M(t)-\scp{F(t)}{Jx(t)}$ and $\cR(t)=\mat{A(t)}{b(t)}{b^\top(t)}{c(t)}$ is the grand resistance matrix of \cite{HappelBrenner1983}, whose entries are given by
\begin{subequations}\label{709a}
\begin{equation}\label{709}
A(t):=\!\int_0^L\! K_\xi(s,t)\,\de s, \quad b(t):=\!\int_0^L\! K_\xi(s,t)J\xi(s,t)\,\de s,
\end{equation}
\begin{equation}\label{709b}
c(t):=\!\int_0^L\! \scp{J\xi(s,t)}{K_\xi(s,t)J\xi(s,t)}\,\de s.
\end{equation}
\end{subequations}

\par It is easy to see that the functions $A$, $b$, and $c$ are ultimately determined by the shape function $\xi$ alone. The terms
\begin{equation}\label{006}
F^\sh(t):=-\int_0^L K_\xi(s,t)\dot\xi(s,t)\,\de s,\qquad M^\sh(t):=-\int_0^L \langle J\xi(s,t),K_\xi(s,t)\dot\xi(s,t)\rangle\,\de s,
\end{equation}
are the contributions to the force and torque due to the shape deformation of the swimmer, and they depend linearly on the time derivative $\dot\xi$.

\par Enforcing the force and torque balance \eqref{105} is equivalent to setting \eqref{1002} equal to zero and solving for $\dot x$ and $\dot\theta$, which eventually leads to the equations
\begin{equation}\label{1027}
\left\{\begin{array}{rl}
\dot x(t) &\!\!\!\! = R(t)v(t),\\ [2pt]
\dot\theta(t) &\!\!\!\! = \omega(t),
\end{array}\right.
\end{equation}
where
\begin{equation}\label{1028}
v(t):=\bar A(t)F^\sh(t)+\bar b(t)M^\sh(t),\qquad \omega(t):=\bar b^\top(t)F^\sh(t)+\bar c(t)M^\sh(t),
\end{equation}
and $\bar A(t)$, $\bar b(t)$, and $\bar c(t)$ are the block elements of the inverse matrix $\cR^{-1}(t)$. The structure of this system of ordinary differential equations is the same as that previously obtained in \cite{ADSL2009,DMDSM11}. The following result, analogous to \cite[Theorem 6.4]{DMDSM11}, holds
\begin{theorem} \label{0001}
Let $\xi\in X_1\,$, $x_0\in\R2$, and $\theta_0\in\R{}$. Then the equations of motion \eqref{1027}, with initial conditions $x(0)=x_0$ and $\theta(0)=\theta_0$, have a unique absolutely continuous solution $t\mapsto(x(t),\theta(t))$ defined in $[0,T]$ with values in $\R2{\times}\R{}$. This solution actually belongs to $\Hp1{}(0,T)$. In other words, there exists a unique rigid motion $t\mapsto r(t)$, such that the deformation function defined by \eqref{1001} belongs to $X_1$, satisfies the equations of motion \eqref{105}, and the initial conditions $x(0)=x_0$ and $R(0)=R_{\theta_0}\,$, the rotation of angle $\theta_0$.
\end{theorem}
\begin{proof}
The result easily follows from the classical theory of ordinary differential equations, see, e.g., \cite{Hale1980}. Indeed, the coefficients $\bar b^\top$ and $\bar c$ are continuous function of $t$, since they come from the inversion of the grand resistance matrix $\cR$, whose entries are continuous in $t$. On the contrary, $F^\sh$ and $M^\sh$ are only $\Lp2{}$ functions of time. This is enough to integrate the second equation in \eqref{1027}. By plugging the solution for $\theta$ into the first equation and by an analogous argument on the coefficients $\bar A$ and $\bar b$, also the equation for $x$ has a unique solution with prescribed initial data.

\par The last statement follows esily from Proposition \ref{752}.
\end{proof}

\par Some notes on the matrix $K$ and on the coefficients $C_\tau$ and $C_\nu$ are in order. First, we assume that 
\begin{equation}\label{716}
0<C_\tau<C_\nu\,,
\end{equation}
secondly, we notice that the matrix $K_\chi$ (and therefore $K_\xi$) is symmetric and positive definite, and defines a scalar product in the space $X_1$\,. Indeed, by introducing the power expended during the motion
\begin{equation}\label{007}
\begin{split}
\cP(\chi):= & \int_0^T\!\!\int_0^L \scp{-f(s,t)}{\dot\chi(s,t)}\,\de s\de t=\int_0^T\!\!\int_0^L \scp{K_\chi(s,t)\dot\chi(s,t)}{\dot\chi(s,t)}\,\de s\de t\\
= & \int_0^T\!\!\int_0^L [C_\tau\dot\chi^2_\tau(s,t)+C_\nu\dot\chi^2_\nu(s,t)]\,\de s\de t,
\end{split}
\end{equation}
we find that
\begin{equation*}
C_\tau\norma{\dot\chi}_{\Lp2{}(0,T;\Lp2{}(0,L))}^2\leq\cP(\chi)\leq C_\nu\norma{\dot\chi}_{\Lp2{}(0,T;\Lp2{}(0,L))}^2\,.
\end{equation*}
Moreover, it follows from \eqref{0000} and \eqref{8360} that the matrices $K_\xi$ and $K_\chi$ are continuous in $(s,t)$.

\par Finally, the strict inequality assumption $C_\tau<C_\nu$ cannot be weakened. Indeed, if we had $C_\tau=C_\nu$\,, then $K_\chi(s,t)$ would be a multiple of the identity matrix and therefore, from \eqref{105a}, we would have
$$0=F(t)=-C_\tau\int_0^L \dot\chi(s,t)\,\de s=-C_\tau\frac{\de}{\de t}\left(\int_0^L \chi(s,t)\,\de s\right),$$
which is expressing that the barycenter does not move as time evolves.

\subsection{The shape function}
We introduce now an important assumption on the shape function $\xi$, called \emph{two disks condition}, which rules out  self-intersections of the swimmer. This hypothesis will be crucial in the proof of the existence of an optimal swimming strategy. The idea underlying this condition is that two distinct points of the swimmer cannot become too close to each other during the motion.

\begin{defin}\label{755}
We say that $\xi\in\Hp2{}(0,L;\R2)$ with $\norm{\xi'(s)}=1$ for every $s\in[0,L]$ satisfies the \emph{two disks condition} with radius $\rho>0$ if the following conditions are satisfied (see Fig.~\ref{f1}):
\begin{itemize}
\item[(a)] for every $s\in[0,L]$ there exist open disks $B^1(s),B^2(s)$ of radius $\rho>0$ such that $B^1(s)\cap B^2(s)=\emptyset$, $\xi(s)\in\cl B{}^1(s)\cap\cl B{}^2(s)$, and $\xi(\sigma)\notin B^1(s)\cup B^2(s)$ for every $\sigma\in[0,L]$;
\item[(b)] there exist open half disks $B^-,B^+$ of radius $2\rho$ centered at $\xi(0)$ and $\xi(L)$, respectively, with diameters normal to $\xi'(0)$ and $\xi'(L)$, respectively, such that $\xi(\sigma)\notin B^-\cup B^+$ for every $\sigma\in[0,L]$.
\end{itemize}
\end{defin}
\par Since $\xi$ is of class $\Cl1{}$ and $\norm{\xi'(s)}=1$, the disks considered in condition (a) are uniquely determined by $\xi(s)$ and $\xi'(s)$. Indeed, they are the disks with centers $\xi(s)\pm\rho J\xi'(s)$ and radius $\rho$. In the sequel we will always assume that 
\begin{equation}\label{758}
B^1(s)=B_\rho(\xi(s)+\rho J\xi'(s)),\qquad B^2(s)=B_\rho(\xi(s)-\rho J\xi'(s)).
\end{equation}

\par The following proposition proves an important consequence of the two disks condition.
\begin{proposition}\label{756}
Let $\xi\in\Hp2{}(0,L;\R2)$ with $\norm{\xi'(s)}=1$ for every $s\in[0,L]$. Assume that $\xi$ satisfies the two disks condition for some radius $\rho>0$. Then $\xi$ is injective on $[0,L]$.
\end{proposition}
\begin{proof}
Assume by contradiction that there exist $s_1<s_2$ such that $\xi(s_1)=\xi(s_2)$. It is easy to see that the two disks condition implies that $\xi'(s_1)=\pm\xi'(s_2)$. Assume that $\xi'(s_1)=\xi'(s_2)$\,, the other case being analogous. Since these derivatives have norm 1, by changing the coordinate system we may assume that $\xi'(s_1)=\xi'(s_2)=e_1$\,, the first vector of the canonical basis. We denote the coordinates of $\xi$ by $\xi_1$ and $\xi_2$ and we set $\alpha:=\xi_1(s_1)=\xi_1(s_2)$ and $\beta:=\xi_2(s_1)=\xi_2(s_2)$. By the Local Inversion Theorem, there exist $\eps>0$, $\delta>0$, and a $\Cl1{}$ function $g:(\alpha-\eps,\alpha+\eps)\to\R{}$ such that for every $s\in(s_1-\delta,s_1+\delta)\cap[0,L]$ we have $\norm{\xi_1(s)-\alpha}<\eps$ and $\xi_2(s)=g(\xi_1(s))$. Let
\begin{equation*}
E^1:=\bigcup_{\norm{s-s_1}<\delta} B^1(s),\qquad E^2:=\bigcup_{\norm{s-s_1}<\delta} B^2(s).
\end{equation*}
\par By \eqref{758} it is easy to see that there exists an open rectangle $R$ centered at $(\alpha,\beta)$ such that $E^1\cap R=\{(a,b)\in R: b>g(a)\}$ and $E^2\cap R=\{(a,b)\in R: b<g(a)\}$. By condition (a) of Definition \ref{755} $\xi(s)\in \{(a,b)\in R: b=g(a)\}$ for every $s\in[0,L]$ such that $\xi(s,t)\in R$. Therefore, $s\mapsto\xi(s)$ is locally an arc length parametrization of the graph of $g$. Since $\xi(s_1)=\xi(s_2)=(\alpha,\beta)$ and $\xi'(s_1)=\xi'(s_2)=e_1$\,, there exists $\eta>0$ such that $\xi(s_1+s)=\xi(s_2+s)$ for $\norm s<\eta$, provided $0\leq s_1+s<s_2+s\leq L$. This implies that $\xi(\sigma-(s_2-s_1))=\xi(\sigma)$ for every $\sigma$ in a neighborhood of $s_2$ in $[0,L]$ such that $\sigma-(s_2-s_1)\in[0,L]$. By taking the supremum $\sigma_0$ over $\sigma$, we obtain that $\xi(\sigma_0)=\xi(L)$ for $\sigma_0:=L-(s_2-s_1)\in[0,L)$. If $\sigma_0>0$, we deduce also that $\xi'(\sigma_0)=\xi'(L)$. The same equality holds when $\sigma_0=0$, because in this case $\sigma_0=s_1=0$ and $s_2=L$, so that 
the equality 
follows from the assumption $\xi'(s_1)=\xi'(s_2)$.

\par Let $B^+$ be the half disk considered in condition (b) of Definition \ref{755}. By the previous equalities, $\xi(\sigma_0)$ is the center of $B^+$ and $\xi'(\sigma_0)$ points towards the interior. It follows that $\xi(\sigma)\in B^+$ for some $\sigma>\sigma_0$\,, and this contradicts condition (b).
\end{proof}

\begin{figure}[!htbp]
\centering
\includegraphics[scale=1]{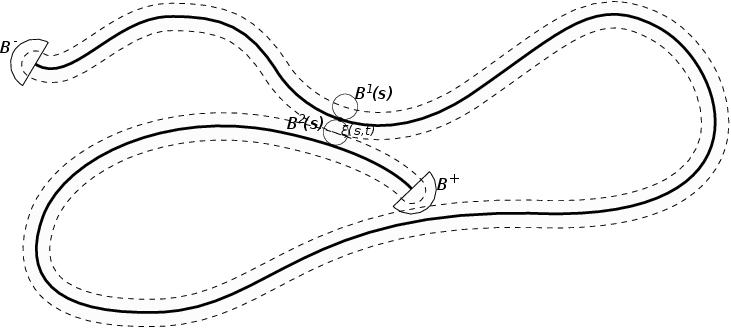}
\caption{Two disks condition: $B^1(s)$ and $B^2(s)$ are the open disks of radius $\rho$, $B^-$ and $B^+$ are the open half disks of radius $2\rho$. The dashed line represents the image of the map $h$ defined in \eqref{1025}.}
\label{f1}
\end{figure}

\par Given $\rho>0$ we introduce $C_{L,\rho}:=([0,L]{\times}\{0\})+B_\rho(0)$, the cigar-like set obtained by enlarging $[0,L]{\times}\{0\}$. 
Define now a map $h\colon [0,L]{\times}(-\rho,\rho)\to\R2$ by
\begin{equation}\label{1025}
h(s,y):=\xi(s)+yJ\xi'(s).
\end{equation}
This map is extended to a continuous map $h\colon C_{L,\rho}\to\R2$ by defining it as an isometry mapping $C_{L,\rho}\cap\{s<0\}$ into the half disk $\frac12B^-$ homothetic to the half disk $B^-$ considered in condition (b) of Definition \ref{755}, with the same center and half the radius; the definition in $C_{L,\rho}\cap\{s>L\}$ is similar and uses the half disk $\frac12B^+$ (see Fig.~\ref{f1}). The following proposition improves Proposition \ref{756} and provides an equivalent formulation of Definition \ref{755}.

\begin{proposition}\label{757}
Let $\xi\in\Hp2{}(0,L;\R2)$ with $\norm{\xi'(s)}=1$ for every $s\in[0,L]$, and let $\rho>0$. The map $h$ is injective if and only if $\xi$ satisfies the two disks condition with radius $\rho$.
\end{proposition}
\begin{proof}
Let us assume that the two disks condition holds and let us consider two points $(s_1,y_1)\neq(s_2,y_2)$ in $C_{L,\rho}$\,. If $0\leq s_1=s_2\leq L$, then it must be $\norm{y_1-y_2}>0$, and therefore $|h(s_1,y_1)-h(s_1,y_2)|=\norm{y_1-y_2}>0$. If $-\rho<s_1=s_2<0$, then $h(s_1,y_1)\neq h(s_1,y_2)$ since $h$ is an isometry on $C_{L,\rho}\cap\{s<0\}$. The same conclusion holds if $L<s_1=s_2<L+\rho$.

\par Assume now that $0\leq s_1<s_2\leq L$. If $y_1=y_2=0$, then $h(s_1,y_1)=\xi(s_1)\neq\xi(s_2)=h(s_2,y_2)$, where the inequality follows from the fact that the curve $\xi(\cdot,t)$ is injective by Proposition \ref{756}. If $y_1\neq0=y_2$\,, then $h(s_1,y_1)$ belongs to one of the disks $B^1(s_1)\,, B^2(s_1)$ introduced in condition (a) of Definition \ref{755}. It follows that $h(s_2,y_2)=\xi(s_2)\notin B^1(s_1)\cup B^2(s_1)$\,, hence $h(s_1,y_1)\neq h(s_2,y_2)$. The same conclusion holds if $y_1=0\neq y_2$\,.

\par Let us consider now the case $0\leq s_1<s_2\leq L$ and $y_1\neq0\neq y_2$\,. Define $S_i:=\{\xi(s_i)+yJ\xi'(s_i):0<(\sign y_i)y<\rho\}$ for $i=1,2$. Let us prove that
\begin{equation}\label{760}
S_1\cap S_2=\emptyset.
\end{equation}
Let $p_i:=\xi(s_i)$ and let $D_i$ be the open disk with center $c_i:=\xi(s_i)+\sign(y_i)\rho J\xi'(s_i)$ and radius $\rho$. Note that $c_i$ and $p_i$ are the endpoints of $S_i$ and that $p_1\neq p_2$ since the map $s\mapsto\xi(s)$ is injective by Proposition \ref{756}. If $c_1=c_2$\,, then $S_1\cap S_2=\emptyset$ because $S_1, S_2$ are radii of the same circle with different endpoints.

\par We consider now the case $c_1\neq c_2$\,. Since $D_1$ is one of the disks $B^1(s_1), B^2(s_1)$, by condition (a) of Definition \ref{755}, we have that $p_1=\xi(s_1)\in\partial D_1$ and $p_2=\xi(s_2)\notin D_1$\,. Similarly, we prove that $p_2\in\partial D_2$ and $p_1\notin D_2$\,. Therefore, $p_1\in\partial D_1\setminus D_2$ and $p_2\in\partial D_2\setminus D_1$\,.

\par Assume by contradiction that $S_1$ and $S_2$ meet at some point $p$, which must belong to $D_1\cap D_2$\,. Let $z$ be the intersection between $\partial D_2$ and the half-line stemming from $c_2$ and containing $c_1$; under our assumptions, we have $z\in D_1$\,. Since $p\in D_1\cap D_2$ and $p_1\in\partial D_1\setminus D_2$\,, there exists a unique point $q\in\partial D_2\cap D_1$ on the segment joining $p$ and $p_1$\,. Now, the half-line through $p$ stemming from $c_2$ meets $\partial D_2$ on the smallest arc $\Gamma$ with endpoints $q$ and $z$. Since $q,z\in D_1$ and the disks have the same radius, we have $\Gamma\subset D_1$ (see Fig.~\ref{f9}). The previous argument shows that $p_2\in\Gamma$, which contradicts the condition $p_2\in\partial D_2\setminus D_1$\,. This concludes the proof of the equality $S_1\cap S_2=\emptyset$ in the case $0\leq s_1<s_2\leq L$ and $y_1\neq0\neq y_2$\,, and implies that $h(s_1,y_1)\neq h(s_2,y_2)$.
\begin{figure}[h]
\centering
\includegraphics[scale=.8]{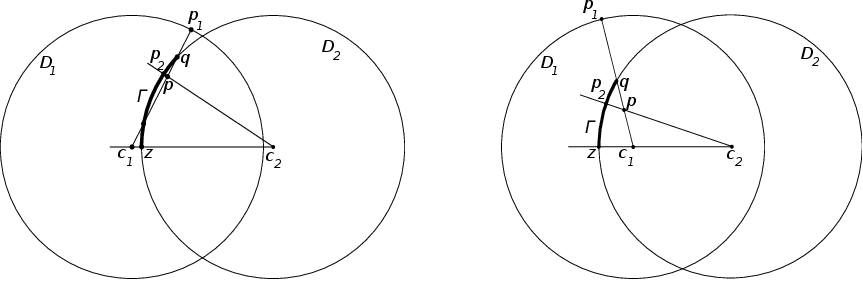}
\caption{Injectivity for $0\leq s_1<s_2\leq L$ and $y_1\neq0\neq y_2$ in the case $c_1\neq c_2$\,: Two possible situations contradicting \eqref{760}.}
\label{f9}
\end{figure}
\par We consider now the case $0\leq s_1< L<s_2<L+\rho$. Assume by contradiction that $h(s_1,y_1)=h(s_2,y_2)=:p$. Observe that $p\in\frac12B^+$. Denote $p_1:=\xi(s_1)$ and $S:=\{\xi(L)+yJ\xi'(L): -\rho<y<\rho\}$. By \eqref{760} for $s_1$ and $L$ the segment with endpoints $p,p_1$ does not intersect $S$. On the other hand $\norm{p-p_1}=\norm{y_1}<\rho$. By elementary geometric arguments we find that the set $Q$ of points which can be connected to a point of $\frac12B^+$ by a segment disjoint from $S$ and of length less than $\rho$ is contained in the union $B^+\cup B^1(L)\cup B^2(L)$ (See Fig.~\ref{f10}). Therefore $p_1=\xi(s_1)\in Q$ and this violates either condition (a) or condition (b) in Definition \ref{755}. 

\par In the case $s_1=L<s_2<L+\rho$, we have $h(s_1,y_1)\in \partial(\frac12 B^+)$, while $h(s_2,y_2)\in\frac12 B^+$, so that $h(s_1,y_1)\neq h(s_2,y_2)$. The cases $-\rho<s_1<0< s_2\leq L$ and $-\rho<s_1<0=s_2$ are analogous.
\begin{figure}[h]
\centering
\includegraphics[scale=1]{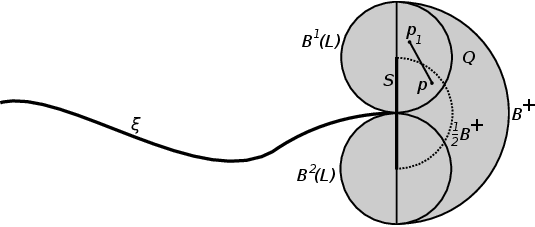}
\caption{Injectivity for $0\leq s_1\leq L<s_2<L+\rho$. The shaded region represents the set $Q$; the thick segment is the set $S$.}
\label{f10}
\end{figure}
\par The last case to consider is when $-\rho<s_1<0$ and $L<s_2<L+\rho$. 
Assume, by contradiction, that $h(s_1,y_1)=h(s_2,y_2)$. 
Since the radius of curvature of $\xi$ is always less than $\rho$, one can prove (see Lemma \ref{r12} below and Fig.\@ \ref{f34}) that
\begin{equation}\label{k10}
D^-\subset B^-\cup\bigcup_{0\le s\le \frac\pi2\rho}\big(B^1(s)\cup\{\xi(s)\}\cup B^2(s)\big)\,,
\end{equation}
where $D^-:=\{x\in \R2: d(x,\frac12 B^-){<\,}\rho\}$ and $B^1(s),\, B^2(s)$ are the open disks defined in \eqref{758}. 
\begin{figure}[!hbt]
\centering
\includegraphics[scale=.3]{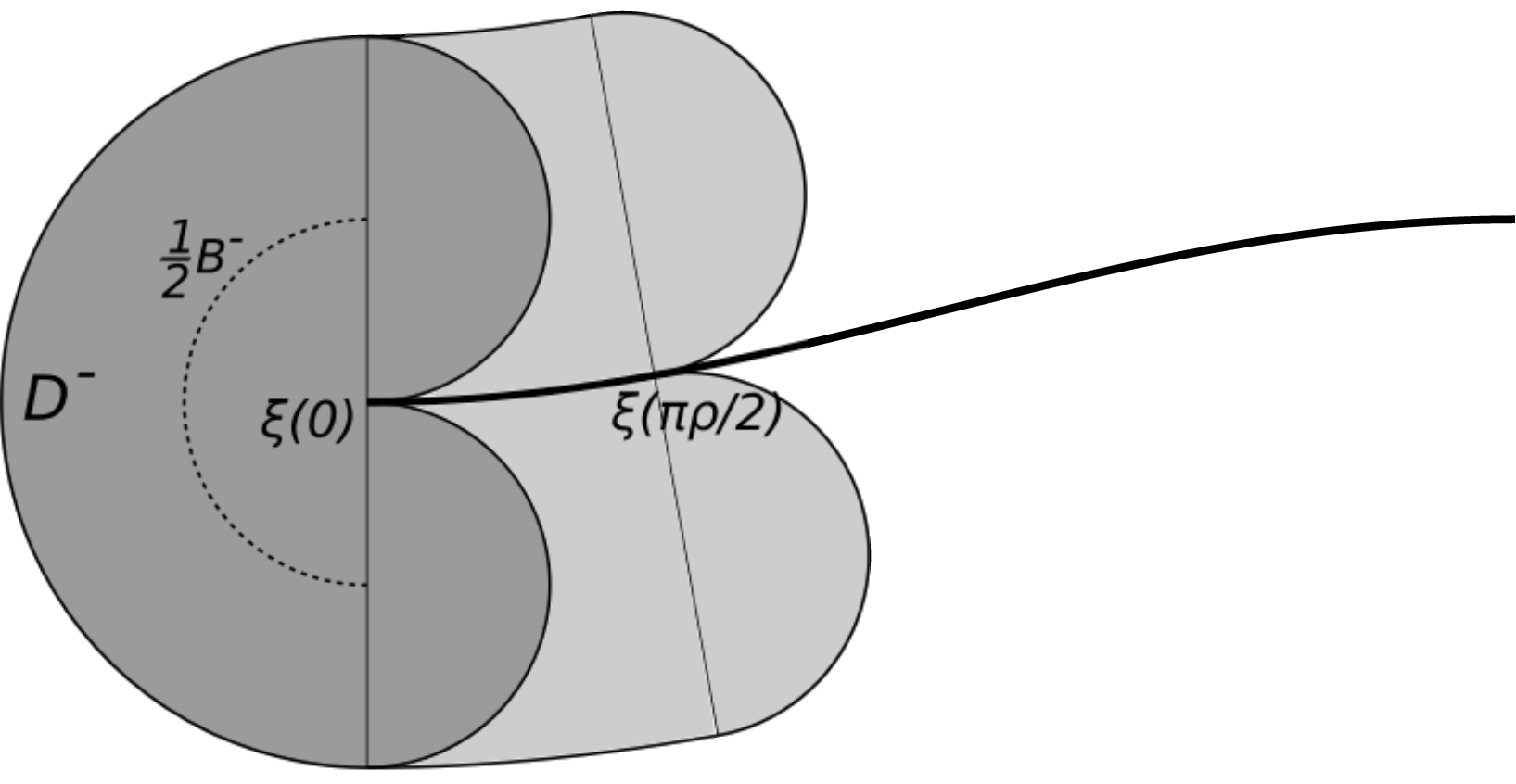}
\caption{The region $D^-$ of \eqref{k10}.}
\label{f34}
\end{figure}

\par Since  $h(s_1,y_1)\in\frac12B^-$, $h(s_1,y_1)=h(s_2,y_2)$, and $|h(s_2,y_2)-\xi(L)|<\rho$, we deduce that $\xi(L)\in D^-$. By \eqref{k10} either $\xi(L)\in B^-$ or there exists $s\in[0,\frac\pi2]$ such that $\xi(L)\in B^1(s)\cup\{\xi(s)\}\cup B^2(s)$. This contradicts Definition \ref{755} or Proposition \ref{756}, and concludes the proof in this case.
\medskip
\par Let us now assume that $h$ is injective, and consider $s\in[0,L]$. 
Let $B^1(s),\, B^2(s)$ the open disks defined in \eqref{758}. It is clear that $B^1(s)\cap B^2(s)=\emptyset$ and $\cl B{}^1(s)\cap\cl B{}^2(s)=\{\xi(s)\}$.
Denote by $N(s):=\{\xi(s)+yJ\xi'(s),y\in\R{}\}$ the normal line through the point $\xi(s)$, and let $\widetilde\xi(s):=\xi(s)+\frac1{\kappa(s)}J\xi'(s)$ be the evolute of $\xi$, i.e., the curve that contains the centers of the osculating circles to $\xi$.
Since $\xi$ is a curve of class $H^2$, the curvature is well defined almost everywhere; let $s_1\in(0,L)$ be a point at which the curvature $\kappa(s_1)$ is defined, and let $s_2\in(0,L)$ be another point.
By the injectivity of $h$, the normal lines $N(s_1)$ and $N(s_2)$ cannot meet at a distance less than $\rho$ from the curve $\xi$, and their intersection tend to $\widetilde\xi(s_1)$ as $s_2\to s_1$ (see \cite[Ex.\@7, page 23]{DoCarmo}).
This implies that $\kappa(s_1)\le 1/\rho$. We conclude that $\kappa(s)\le 1/\rho$ for a.e.\ $s\in[0,L]$.
Let $I_\rho(s):=\{\sigma\in[0,L]:|\sigma-s|<\rho\}\cap[0,L]$ be a neighborhood of $s$ of radius $\rho$ in $[0,L]$.
By standard results in Differential Geometry, $\xi(\sigma)\notin B^1(s)\cup B^2(s)$ for every $\sigma\in I_\rho(s)$. 
If now $\sigma\in[0,L]\setminus I_\rho(s)$, then again $\xi(\sigma)\notin B^1(s)\cup B^2(s)$. Indeed, assume by contradiction that
there exists $\sigma\in[0,L]\setminus I_\rho(s)$ such that $\xi(\sigma)\in B^1(s)$. Then there exists $r\in(0,\rho)$ such that
$\xi(\sigma)\in B_r(\xi(s)+rJ\xi'(s))$.
Then, let $y_1:=\min_{\sigma\in[0,L]}\{|\xi(s)+rJ\xi'(s)-\xi(\sigma)|\}<r$, and let $s_1\in[0,L]\setminus I_\rho(s)$ be the point where the minimum is attained. If $s_1\in(0,L)$, then
$$\text{either}\qquad\xi(s)+r J\xi'(s)=\xi(s_1)+y_1J\xi'(s_1)\qquad\text{or}\qquad\xi(s)+r J\xi'(s)=\xi(s_1)-y_1J\xi'(s_1).$$
This violates the injectivity of $h$. 
The cases $s_1 = 0$ and $s_1=L$ lead to a similar contradiction, taking into account the definition of $h$ near the endpoints of the segment. 
This proves that condition (a) in Definition \ref{755} holds.

\par To prove that condition (b) holds,  we assume, by contradiction, that there exists $s_0\in[0,L]$ such that $\xi(s_0)\in B^-$.
We first observe that, for $s\in[0,\pi\rho]$, the point $\xi(s)$ lies on the opposite side of $B^-$ with respect to its diameter $\{\xi(0)+yJ\xi'(0):-\rho<y<\rho\}$, so that 
$\xi(s)\notin B^-$ for $s\in[0,\pi\rho]$. Then $s_0$ belongs be the closed set $\Sigma$ of points $s\in[\pi\rho,L]$ such that $\xi(s)$ lies in the closure of 
$B^-$. Let $s_1$ be the minimum point of 
$$
\min_{s\in\Sigma} |\xi(s)-\xi(0)|.
$$
Since $\xi(s)\notin B^1(0)\cup B^2(0)$ for every $s\in [0,L]$, the point $\xi(s_1)$ does not belong to the diameter $\{\xi(0)+yJ\xi'(0):-\rho<y<\rho\}$. Therefore $\xi(s_1)\in B^-$ and $s_1$ belongs to the interior of $\Sigma$. This implies that $\xi(s_1)-\xi(0)$ is orthogonal to $\xi'(s_1)$, hence $\xi(s_1)-\xi(0)=yJ\xi'(s_1)$ for some $y\in (-2\rho,2\rho)$. Then $h(s_1,-\frac y2)=\xi(s_1)-\frac y2J\xi'(s_1)\in \frac12 B^-$. This violates the injectivity of $h$ and concludes the proof of the proposition.
\end{proof}
\begin{lemma}\label{r12}
Let $\rho>0$ and let $\xi\in\Hp2{}(0,\frac\pi2\rho;\R2)$ with $\norm{\xi'(s)}=1$ for every $s\in[0,\frac\pi2\rho]$ and $\norm{\xi''(s)}\le 1/\rho$  for a.e.\ $s\in[0,\frac\pi2\rho]$. Then \eqref{k10} holds.
\end{lemma}

\begin{proof} Let $\xi_1(s)$ and $\xi_2(s)$ be the coordinates of $\xi(s)$ and let $\vartheta(s)$ be the oriented angle between $e_1:=(1,0)$ and $\xi'(s)$.
It is well known that $|\vartheta'(s)|= \norm{\xi''(s)}\le 1/\rho$. Assume, for simplicity, that $\xi(0)=0$ and $\xi'(0)=e_1$.
By standard results in Differential Geometry the curve $\xi$ does not intersect the open disks $B_\rho(\pm\rho\, e_2)$, where $e_2:=(0,1)$. By integration we obtain
$$
\xi_1(s)=\int_0^s \cos\vartheta(\sigma)\,\de \sigma\,.
$$
Let us prove that
\begin{equation}\label{k15}
\rho \sin ( s/\rho) \le \xi_1(s)
\end{equation}
for every $s\in[0,\frac\pi2]$. Since the inequality is true for $s=0$, it is enough to show that  $\cos(s/\rho)\le \xi_1^\prime(s)=\cos\vartheta(s)=
\cos|\vartheta(s)|$ for every $s\in[0,\frac\pi2]$. Since $\xi'(s)=(\cos\vartheta(s), \sin\vartheta(s))$, we have $|\vartheta'(s)|\le 1/\rho$, hence $|\vartheta(s)|\le s/\rho$, so that $\cos(s/\rho)\le \cos|\vartheta(s)|$ by the monotonicity of $\cos$ in $[0,\frac\pi2]$. This concludes the proof of 
\eqref{k15}.  

Let $S$ be the segment with endpoints $(\rho,\pm \rho)$, which belong to the circles $\partial B_\rho(\pm\rho\, e_2)$. Inequality \eqref{k15} implies that  the curve $\xi$ intersects the segment $S$.
Since the bound on the curvature implies that $\xi$ cannot have a vertical tangent, except when $\xi$ is contained in $\partial B_\rho(\pm\rho\, e_2)$, the intersection point is unique.
Let $s_0\in [0,\rho\frac\pi2]$ be the value of the arc length parameter of this intersection point and let  $N(s_0):=\{\xi(s_0)+yJ\xi'(s_0),y\in\R{}\}$ be the corresponding normal line to $\xi$.

If $\xi(s_0)=(\rho,\pm\rho)$, then the bound on the curvature implies that  $\xi$ is contained in $\partial B_\rho(\pm\rho\, e_2)$ and the statement of the lemma is easily checked.
So we may assume that $\xi(s_0)\neq (\rho,\pm\rho)$. We may also assume that $\xi_2'(s_0)\ge 0$. If not, we just reverse the orientation of the $x_2$-axis.

Let $B^1(s_0)$ be the tangent disk to $\xi$ at $\xi(s_0)$ defined by $B^1(s_0):=B_\rho(\xi(s_0)+\rho J\xi'(s_0))$. We claim that
\begin{equation}\label{k16}
(\rho, \rho )\in B^1(s_0)\,.
\end{equation}
If $\xi_2'(s_0)= 0$, then $(\rho, \rho )\in N(s_0)$ and its distance from $\xi(s_0)$ is less than $\mathrm{length}(S)=2\rho$, which implies \eqref{k16}.
If $\xi_2'(s_0)\neq 0$, we argue by contradiction. Assume that \eqref{k16} is not satisfied. 
Then $\partial B^1(s_0)$ intersects $S$ in $\xi(s_0)$ and in another point $\tilde\xi$ between $\xi(s_0)$ and $(\rho, \rho )$.
Therefore the center $c_1$ of the disk $B^1(s_0)$ is the vertex of an isosceles triangle with basis contained in $S$ and equal sides of length $\rho$.
Elementary geometric arguments show that this vertex must belong to the astroid obtained by removing the four disks $B_\rho((\pm \rho,\pm \rho))$ from the square $(-\rho,\rho)\times(-\rho,\rho)$  (see Fig.~\ref{f33}).
\begin{figure}[htb]
\centering
\includegraphics[scale=.25]{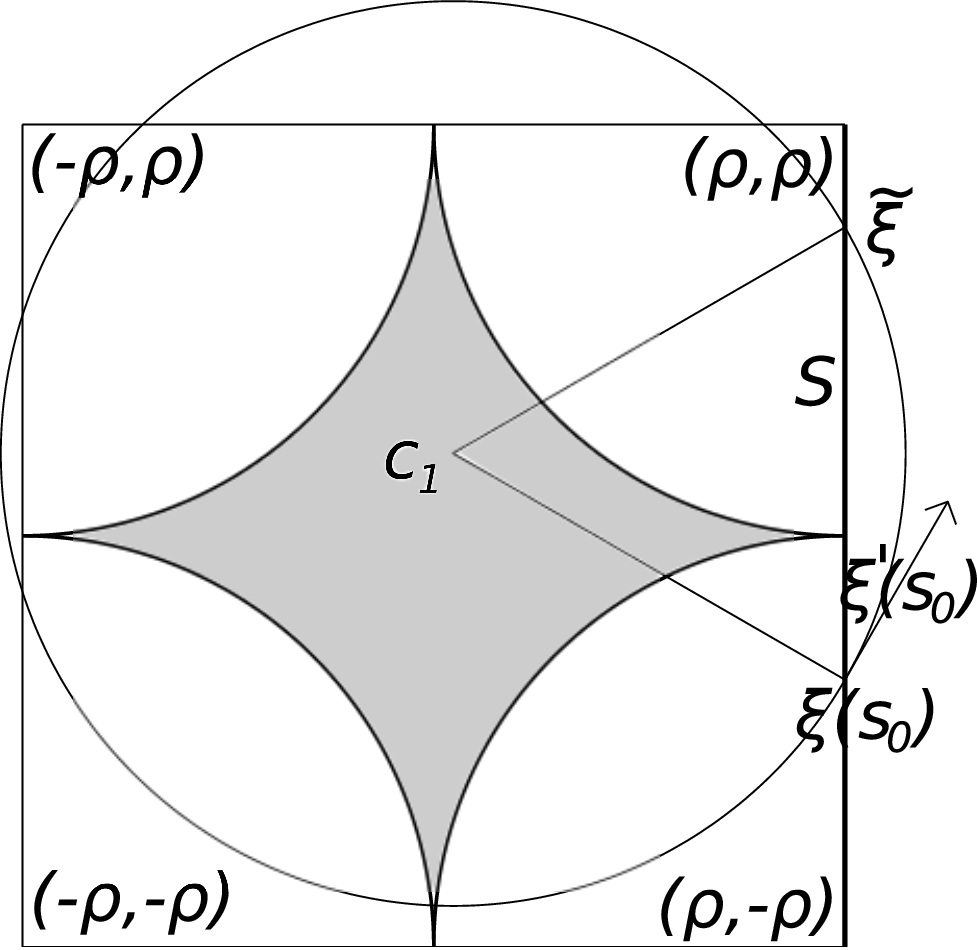}
\caption{The absurd situation in the proof by contradiction. The circle is $B^1(s_0)$, while the curve is not shown.}
\label{f33}
\end{figure}
Therefore the distance from the origin of the center of $B^1(s_0)$ is less than $\rho$. This implies that $0\in B^1(s_0)$.
On the other hand, since $\xi$ is tangent to this disk at $\xi(s_0)$, the bound on the curvature implies that $\xi(s)\notin B^1(s_0)$ for every $s\in[0,\frac\pi2]$.
This contradicts the assumption $\xi(0)=0$ and concludes the proof of \eqref{k16}.

Let $C$ be the curvilinear triangle obtained by removing the disks $B_\rho(\pm\rho\, e_2)$ from the rectangle $(0,\rho)\times(-\rho,\rho)$.
Let $C^+$ and $C^-$ be the parts of $C$ weakly above and strictly below $N(s_0)$ and let $p$ be the intersection of $N(s_0)$ and $\partial B_\rho(\rho\, e_2)$ contained in the closure of $C$.
Since the distance between $p$ and $\xi(s_0)$ is less than $2\rho$, we deduce that $p\in B^1(s_0)$. 
Since $(\rho,\rho)\in B^1(s_0)$ and $\xi(s_0)\in \partial B^1(s_0)$, we obtain that $C^+$ is contained in $B^1(s_0)$.

Let us prove that
\begin{equation*}
C^-\subset \bigcup_{0\le s\le s_0}\big(B^1(s)\cup\{\xi(s)\}\cup B^2(s)\big)
\end{equation*}
Let us fix $x\in C^-$ and let $s_1$ be a minimizer of
$$
\min_{0\le s\le s_0}|x-\xi(s)|.
$$
Since $\scp{x}{\xi'(0)}>0$ and $\scp{x-\xi(s_0)}{\xi'(s_0)}<0$, we have $s_1\in (0,s_0)$, hence the orthogonality condition
$$
\scp{x-\xi(s_1)}{\xi'(s_1)}=0.
$$
Since $|x-\xi(s_1)|<|x|<2\rho$, the point $x$ belongs to  $B^1(s_1)\cup\{\xi(s_1)\}\cup B^2(s_1)$, which concludes the proof.
\end{proof}

We now prove a result stating that a bound on the angle $\vartheta$ formed by the tangent with the first axis implies the non self-intersection of the swimmer.
\begin{lemma}\label{110}
Let $\vartheta
\in\Cl1{}([0,L])$, let $\kappa_0:=\max\left\{|\vartheta'(s)|:s\in[0,L]\right\}$, and let
\begin{equation*}
\xi(s)=\int_0^s \vec{\cos\vartheta(\sigma)}{\sin\vartheta(\sigma)}\de \sigma.
\end{equation*}
Assume $\norm{\vartheta
(s)}<\pi/4$ for every $s\in[0,L]$. Then, $\xi$ satisfies the two disks condition with radius $\rho$, for every $0<\rho\leq 1/\kappa_0$.
\end{lemma}
\begin{proof}
Notice, in the first place, that $|\xi'(s)|=1$ for every $s\in[0,L]$, so that $\xi$ is a regular curve parametrized by arc length. 
The condition $|\vartheta(s)|<\pi/4$ for all $s\in[0,L]$ implies that $\xi$ is a graph with respect to the $x_1$-axis. 
Given $\rho$, with $0<\rho\leq 1/\kappa_0$, we define the open disks $B^1(s):=B_{\rho}(\xi(s)+\rho J\xi'(s))$ and $B^2(s):=B_{\rho}(\xi(s)-\rho J\xi'(s))$, as in \eqref{758}. 
Since $\xi$ is a graph and its curvature is bounded by $\kappa_0$, the disks
 $B^1(s)$ and $B^2(s)$ satisfy condition (a) in Definition \ref{755}. 
Finally, the construction of $B^-$ and $B^+$ as in condition (b) of Definition \ref{755} is straightforward.
\end{proof}
The preceding lemma will be useful in Section \ref{controllability} to check that the deformations we construct to prove the controllability of the swimmer are admissible.

\section{Controllability}\label{controllability}
\par In this Section we show that the swimmer is controllable, i.e., it is possible to prescribe a self-propelled motion that takes it from a given initial state $\chi_\init$
to a given final state $\chi_\fin$\,.
More precisely, we prove the following theorem.
\begin{theorem}\label{300}
Let $\rho>0$ and let $\chi_\init$\,, $\chi_\fin\in H^2(0,L;\R2)$, with  $|\chi_\init'(s)|=|\chi_\fin'(s)|=1$ for every $s\in[0,L]$. 
Assume that $\chi_\init$ and $\chi_\fin$ satisfy the two disks condition with radius $\rho$ (see Definition \ref{755}).
Then, there exists $\chi \in X_1$\,, 
satisfying the force and torque balance \eqref{105}, such that $\chi(s,0)=\chi_\init(s)$ and $\chi(s,T)=\chi_\fin(s)$ for every $s\in[0,L]$, 
and such that for every $t\in[0,T]$ the curve $\chi(\cdot,t)$ satisfies the two disks condition with radius $\rho$.
\end{theorem}
\begin{proof}
To construct $\chi$, we divide the interval $[0,T]$ into three intervals $[0,\frac13T]$, $[\frac13T,\frac23T]$, and $[\frac23T,T]$. 
In the first interval we straighten $\chi_\init$\,, i.e., we construct $\chi$, satisfying the force and torque balance \eqref{105} 
and the two disks condition with radius $\rho$ on $[0,\frac13T]$, such that $\chi(s,0)=\chi_\init(s)$ and $\chi(s,\frac13T)=\Sigma _\init(s)$ 
for every $s\in[0,L]$, where $\Sigma_\init$ is the arc length parametrization of a suitable segment of length $L$, depending on $\chi_\init$\,.

The same construction, with time reversed, shows that there exists a segment $\Sigma_\fin$, depending on $\chi_\fin$\,, that can be
transferred onto $\chi_\fin$\,, i.e., there exists $\chi$ satisfying the force and torque balance \eqref{105} and the the two disks condition with radius $\rho$ on $[\frac23T,T]$, such that $\chi(s,\frac23T)=\Sigma _\fin(s)$ and $\chi(s,T)=\chi_\fin(s)$ for every $s\in[0,L]$.
 
Since, in general, $\Sigma _\fin$ does not coincide with $\Sigma _\init$\,, we use the interval $[\frac13T,\frac23T]$ to transfer $\Sigma _\init$ onto $\Sigma _\fin$\,.

We now describe the construction of $\chi$ on $[0,\frac13T]$. 
First of all, it is possible to find $\xi\in X_1$ such that  $\xi(s,0)=\chi_\init(s)$  for every $s\in[0,L]$ and $s\mapsto \xi(s,\frac13T)$ is affine on $[0,L]$. 
It is also possible to obtain that $\xi(\cdot,t)$ satisfies the two disks condition with radius $\rho>0$ for every $t\in[0,\frac13T]$.

The last requirement can be fulfilled in the following way. If at one end of the swimmer there is enough room, we pull it along the tangent and unwind it from its original shape obtaining a straight configuration, as illustrated in Fig.~\ref{f11}.
\begin{figure}[h]
\centering
\includegraphics[scale=1]{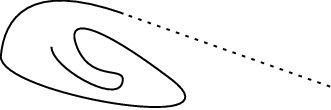}
\caption[Straightening the swimmer I.]{Straightening the swimmer I. The dashed line is the straightened configuration.}
\label{f11}
\end{figure}
\par If this is not the case, then we operate as in Fig.~\ref{f12}: the unwinding is achieved by pinching a point with maximal $x_1$-coordinate and pulling it to the right respecting the two disks condition.
\begin{figure}[ht]
\centering
\includegraphics[scale=.75]{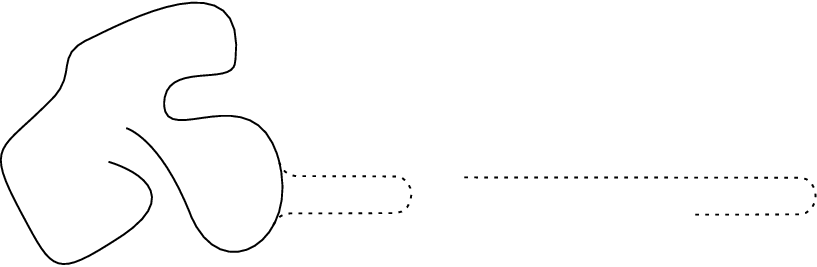}
\caption[Straightening the swimmer II.]{Straightening the swimmer II. The dashed lines represent the intermediate phases of the stretching procedure.}
\label{f12}
\end{figure}

We now compose $\xi$ with a time dependent rigid motion and define $\chi$ on $[0,L]\times[0,\frac13T]$ by \eqref{1001}. Clearly the curve
$\chi(\cdot,t)$ continues to satisfy the two disks condition with radius $\rho$ for every $t\in [0,\frac13T]$. Moreover the function
$\Sigma _\init(s):=\chi(s,\frac13T)$ is affine and $|\Sigma _\init'(s)|=1$ for every $s\in[0,L]$.
The vector $x$ and the rotation $R$ are chosen so that the equation of motion \eqref{1027} is satisfied (this is possible thanks to Theorem \ref{0001}), so that $\chi$ satisfies the force and torque balance \eqref{105}  in $[0,\frac13T]$.

Note that, while the affine map $\xi(\cdot,\frac13T)$ can be chosen freely, the corresponding map $\Sigma _\init$ depends on the superimposed rigid motion, which, in turn, depends on the data of the problem. Therefore, in this construction the location of the segment $\Sigma _\init$ cannot be prescribed.

On $[\frac23T, T]$ the function $\chi$ is defined in a similar way. To transfer $\Sigma _\init$ into $\Sigma _\fin$ in the time interval $[\frac13T,\frac23T]$ we show that, for a straight swimmer, it is possible to produce self-propelled motions achieving any prescribed translation along its axis and any prescribed rotation about its barycenter.

\par To summarize, the whole control process is organized as in Fig. \ref{fig2}.
\begin{equation*}
\chi_\init(\cdot)\xrightarrow{\text{straightening}}\Sigma_\init(\cdot)\xrightarrow{\text{rotation, translation, rotation}} \Sigma_\fin(\cdot)\xrightarrow{\text{straightening$^{-1}$}}\chi_\fin(\cdot).
\end{equation*}

\begin{figure}[ht]
\begin{center}
\centering{
\labellist
	\hair 2pt
	\pinlabel $\chi_\fin$ at 580 50 
	\pinlabel $\chi_\init$ at 18 70 
	\pinlabel $\Sigma_\fin$ at 395 100 
	\pinlabel $\Sigma_\init$ at 200 0 
	\pinlabel $\downarrow$ at 90 62 
	\pinlabel $\longrightarrow$ at 460 90 
	\pinlabel $\longrightarrow$ at 260 45 
	\endlabellist				
\includegraphics[scale=.6]{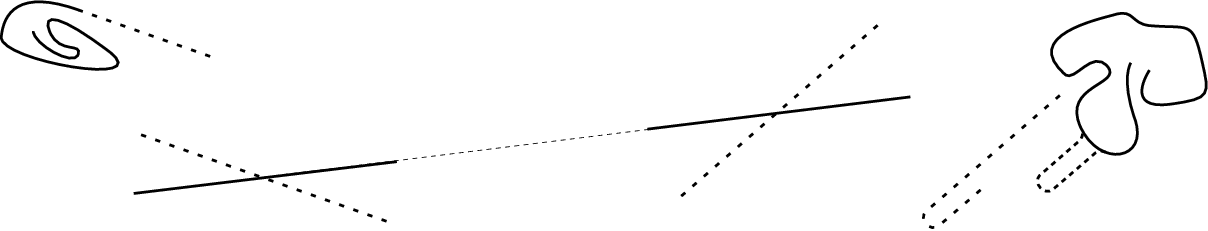}\\ 
}
\end{center}
\caption{Sketch of the control process.}
\label{fig2}
\end{figure}
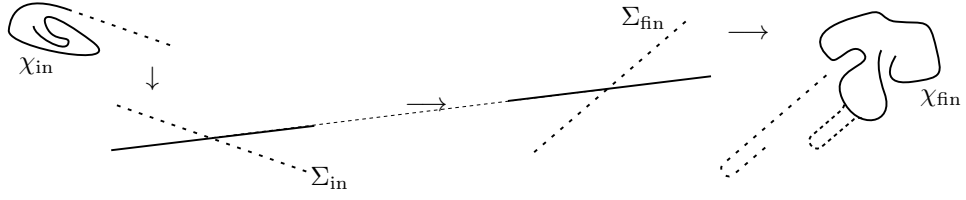

\subsection{Translation}\label{translation}
\par In this subsection we describe how to translate a straight swimmer along its axis: since the problem is rate independent (see Remark \ref{rem}), it is not restrictive to work in the time interval $[0,1]$. The motion of the swimmer is obtained through the translation along the swimmer itself of a localized bump. In order to get a rectilinear motion, we have to assume that the bump satisfies some symmetry properties \eqref{0002}.

\par We can assume that the swimmer lies initially on the $x_1$-axis and that the initial parame\-trization is $\Sigma_\init(s)=se_1$\,.
Given $a\in\R{}$, we describe a self-propelled motion that transfers the segment $\Sigma_\init(s)=se_1$ into the segment $\Sigma_\fin(s)=(a+s)e_1$\,, $s\in[0,L]$. It is not restrictive to assume $a>0$.

\par As before, the motion will be first described through a function $\xi\in X_1$ satisfying the two disks condition with the prescribed radius $\rho$. Then, $\xi$ will be composed with a time dependent rigid motion in order to obtain $\chi$ satisfying also the force and torque balance.

\par For simplicity, we assume $\xi(0,t)=0$ for every $t\in[0,1]$. The function $\xi(s,t)$ will be better described by means of the angle $\vartheta(s,t)$ between its tangent line and the positive $x_1$-axis. This leads to the formula
\begin{equation}\label{121}
\xi(s,t)=\int_0^s \vec{\cos\vartheta(\sigma,t)}{\sin\vartheta(\sigma,t)}\de \sigma.
\end{equation}

\par The function $\vartheta(s,t)$ will be defined using a smooth function $\vartheta_0\colon \R{}\to(-\pi/4,\pi/4)$, with $\max_{s}|\vartheta_0'(s)|\leq1/\rho$ and support 
\begin{equation}\label{ts12}
\spt\vartheta_0=[-\ell,\ell],\qquad\text{where $\ell\in(0,L/2)$.}
\end{equation}
Given two Lipschitz continuous control functions $u_1:[0,1]\to[-1,1]$ and $u_2:[0,1]\to[\ell,L-\ell]$, we define 
\begin{equation}\label{120}
\vartheta(s,t):=u_1(t)\vartheta_0(s-u_2(t)),
\end{equation}
for every $(s,t)\in[0,L]{\times}[0,1]$.
Notice that for all $t\in[0,1]$ the function $\vartheta(\cdot,t)$ has support contained in $[u_2(t)-\ell,u_2(t)+\ell]$, which in turn is contained in $[0,L]$, and the curve $\xi(\cdot,t)$ satisfies the two disks condition with radius $\rho$ (see Lemma \ref{110}).
The parameters $u_1(t)$ and $u_2(t)$ represent the amplitude of the bump and the location of its midpoint at time $t$.

\par It is convenient to introduce the function 
\begin{equation*}\label{r100}
\xi_0(s,u_1,u_2):=\int_0^s \vec{\cos(u_1\vartheta_0(\sigma-u_2))}{\sin(u_1\vartheta_0(\sigma-u_2))}\,\de \sigma,
\end{equation*}
so that 
\begin{equation}\label{r101}
\xi(s,t)=\xi_0(s,u_1(t),u_2(t)).
\end{equation}
It follows that
\begin{equation*}\label{r102}
\xi'(s,t)=\xi_0'(s,u_1(t),u_2(t)),
\end{equation*}
where
\begin{equation}\label{r103}
\xi_0'(s,u_1,u_2)=\frac{\partial\xi_0}{\partial s}(s,u_1,u_2)=\vec{\cos(u_1\vartheta_0(s-u_2))}{\sin(u_1\vartheta_0(s-u_2))}.
\end{equation}
Therefore,
\begin{equation*}\label{r104}
\begin{split}
K_\xi(s,t) = & C_\tau\xi'(s,t)\otimes\xi'(s,t)+C_\nu(J\xi'(s,t))\otimes(J\xi'(s,t)) \\
= & C_\tau\xi_0'(s,u_1(t),u_2(t))\otimes\xi_0'(s,u_1(t),u_2(t)) \\  
& +C_\nu(J\xi_0'(s,u_1(t),u_2(t)))\otimes(J\xi_0'(s,u_1(t),u_2(t))).
\end{split}
\end{equation*}

\par It is convenient to introduce
\begin{equation}\label{r105}
K_0(s,u_1,u_2):=C_\tau\xi_0'(s,u_1,u_2)\otimes\xi_0'(s,u_1,u_2)+C_\nu(J\xi_0'(s,u_1,u_2))\otimes(J\xi_0'(s,u_1,u_2)),
\end{equation}
so that
\begin{equation}\label{r106}
K_\xi(s,t)=K_0(s,u_1(t),u_2(t)).
\end{equation}
We also have
\begin{equation*}\label{r107}
\dot\xi(s,t)=\xi_{0,1}(s,u_1(t),u_2(t))\dot u_1(t)+\xi_{0,2}(s,u_1(t),u_2(t))\dot u_2(t),
\end{equation*}
where
\begin{equation*}\label{r108}
\xi_{0,i}(s,u_1,u_2):=\frac{\partial\xi_0}{\partial u_i}(s,u_1,u_2).
\end{equation*}
A simple computation leads to
\begin{equation}\label{r109}
\xi_{0,1}(s,u_1,u_2)=\int_0^s \vec{-\sin(u_1\vartheta_0(\sigma-u_2))}{\cos(u_1\vartheta_0(\sigma-u_2))}\vartheta_0(\sigma-u_2)\,\de \sigma,
\end{equation}
\begin{equation}\label{r110}
\begin{split}
\xi_{0,2}(s,u_1,u_2)= & \int_0^s \vec{\sin(u_1\vartheta_0(\sigma-u_2))}{-\cos(u_1\vartheta_0(\sigma-u_2))}u_1\vartheta_0'(\sigma-u_2)\,\de \sigma \\
= & -\int_0^s \frac{\de}{\de\sigma}\vec{\cos(u_1\vartheta_0(\sigma-u_2))}{\sin(u_1\vartheta_0(\sigma-u_2))}\,\de \sigma \\
= & \vec{1-\cos(u_1\vartheta_0(s-u_2))}{-\sin(u_1\vartheta_0(s-u_2))},
\end{split}
\end{equation}
for every $s\in[0,L]$, $u_1\in[-1,1]$, and $u_2\in[\ell,L-\ell]$. Note that in the previous computation we have used the fact that $\vartheta_0(-u_2)=0$, since $\spt\vartheta_0\subseteq[-\ell, \ell]$.

\par Finally, we make the following symmetry assumption on the angle function $\vartheta_0$:
\begin{subequations}\label{0002}
\begin{eqnarray}
&&\vartheta_0\text{ is odd in $[-\ell,\ell]$}; \label{0002a} \\
&&\vartheta_0\text{ is even in $[-\ell,0]$ and in $[0,\ell]$}; \label{0002b} \\
&&\vartheta_0\text{ is odd in $[-\ell,-\ell/2]$, $[-\ell/2,0]$, $[0,\ell/2]$, and in $[\ell/2,\ell]$}. \label{0002c}
\end{eqnarray}
\end{subequations}
We say that a function $u\colon[a,b]\to\R{}$ is said to be even (resp.\@ odd) in $[a,b]$ if $u(x)=u(a+b-x)$ (resp.\@ $u(x)=-u(a+b-x)$) for every $x\in[a,b]$.
\par Figure \ref{f3} shows an example of a bump $\xi_0(s,t)$ whose angle function $\vartheta_0$ enjoys the properties listed above; notice that the parity of the vertical component of $\xi_0$ is reversed with respect to that of $\vartheta_0$\,.
\begin{figure}[htbp]
\centering
\includegraphics[scale=.5]{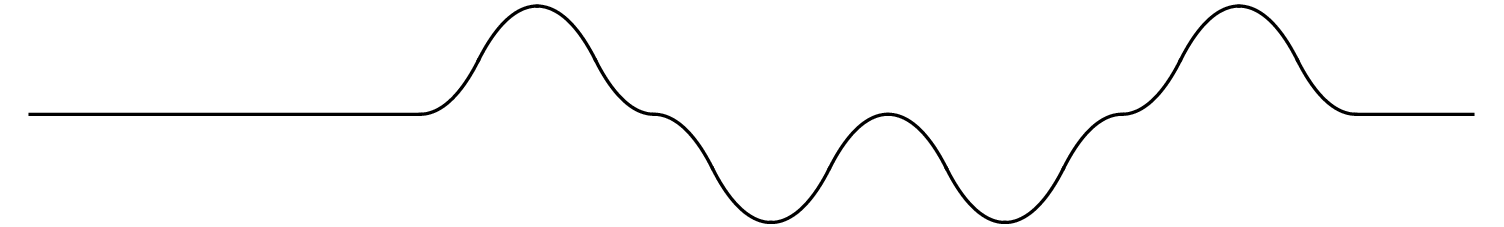}
\caption{Image of the function $\xi_0(s,t)$ for the translational motion.}
\label{f3}
\end{figure}

To exploit these symmetry properties, we repeatedly use the following lemma, whose elementary proof is omitted.
\begin{lemma}\label{220}
Let $u\colon[a,b]\to\R{}$ be an even (resp.\@ odd) function and let $c:=(a+b)/2$ be the middle point of $[a,b]$. Then, the integral function $U(x):=\int_c^x u(s)\,\de s$ is odd (resp.\@ even) in $[a,b]$.
\end{lemma}

\par We now compose $\xi$ defined in \eqref{r101} with a time dependent rigid motion and define $\chi$ on $[0,L]{\times}[0,1]$ by \eqref{1001}. 
The vector $x(t)$ and the rotation $R(t)$ are chosen so that $x(0)=0$ and $R(0)=I$ and the force and torque balance \eqref{105} is satisfied by $\chi$ at every time. 
We want to prove that this is possible with 
\begin{equation}\label{r111}
x_2(t)=0\quad\text{and}\quad R(t)=I, \qquad\text{for all $t\in[0,1]$},
\end{equation}
by a suitable choice of $x_1(t)$.
We shall see that this result follows from the symmetry assumptions \eqref{0002} and does not depend on the particular choice of the control functions $u_1(t),u_2(t)$.

\par Note that for every $t\in[0,1]$ we have $\spt\vartheta(\cdot,t)=[u_2(t)-\ell,u_2(t)+\ell]$. 
Since $\vartheta(s,t)=0$ for $s\in[0,u_2(t)-\ell]$, we obtain from \eqref{121} that $\xi(s,t)=se_1$ for $s\in[0,u_2(t)-\ell]$. 
Similarly, since $\vartheta(s,t)=0$ for every $s\in[u_2(t)+\ell,L]$, \eqref{121} implies that in this interval $\xi(\cdot,t)$ is the arc length parametrization of a segment parallel to the $x_1$-axis.
Therefore, the curve $\xi(\cdot,t)$ is the union of two segments and a connecting bump, corresponding to the restriction of the curve $\xi(\cdot,t)$ to the interval $[u_2(t)-\ell,u_2(t)+\ell]$. 

\par Notice that \eqref{0002a} yields $\int_{-\ell}^\ell \sin(u_1(t)\vartheta_0(s))\,\de s=0$, therefore, \eqref{121} and \eqref{120} imply that for every $t\in[0,1]$ the curve $s\mapsto\xi(s,t)$, $s\in[u_2(t)+\ell,L]$, parametrizes a segment lying on the $x_1$-axis.
Moreover, by a change of variables we have
$$\xi(u_2(t)+\ell,t)-\xi(u_2(t)-\ell,t)=\int_{-\ell}^\ell \vec{\cos(u_1(t)\vartheta_0(s))}{\sin(u_1(t)\vartheta_0(s))}\de s=
\left(\int_{-\ell}^\ell\cos(u_1(t)\vartheta_0(s))\,\de s\right)e_1.$$
These two remarks imply that
\begin{equation*}\label{904}
\xi(s,t)=\left(s-2l+\int_{-\ell}^\ell\cos(u_1(t)\vartheta_0(s))\,\de s\right)e_1,\qquad\text{for }u_2(t)+\ell\leq s\leq L.
\end{equation*}

\par Using \eqref{1001} and \eqref{r111}, the expression for $\chi$ reads $\chi(s,t)=x_1(t)e_1+\xi(s,t)$. 
From this we get $\chi'(s,t)=\xi'(s,t)$ and $\dot\chi(s,t)=\dot x_1(t)e_1+\dot\xi(s,t)$. It follows that the matrix $K_\chi$ defined in \eqref{0000} satisfies
\begin{equation*}\label{1111}
K_\chi(s,t)=K_\xi(s,t)=C_\tau\xi'(s,t)\otimes\xi'(s,t)+C_\nu(J\xi'(s,t))\otimes(J\xi'(s,t)).
\end{equation*}
The linear densities of force and moment (see \eqref{104}) are then given by
\begin{equation*}
\begin{split}
-f(s,t) = &\,\, K_\chi(s,t)\dot\chi(s,t)=\dot x_1(t)K_\xi(s,t)e_1+K_\xi(s,t)\dot\xi(s,t) \\
-m(s,t) = &\,\, \dot x_1(t)\langle K_\xi(s,t)e_1,J\xi(s,t)\rangle+\dot x_1(t)x_1(t)\scp{K_\xi(s,t)e_1}{e_2} \\
& \,\, +\langle K_\xi(s,t)\dot\xi(s,t),J\xi(s,t)\rangle+x_1(t)\langle K_\xi(s,t)\dot\xi(s,t),e_2\rangle
\end{split}
\end{equation*}
By plugging this information in \eqref{105}, we get
\begin{equation}\label{139}
\dot x_1(t)\int_0^L K_\xi(s,t)e_1\,\de s=-\int_0^L K_\xi(s,t)\dot\xi(s,t)\,\de s,
\end{equation}
where the right-hand side is $F^\sh(t)$ defined in \eqref{006}, and
\begin{equation}\label{921}
\begin{split}
\dot x_1(t)&\int_0^L \langle K_\xi(s,t)e_1,J\xi(s,t)\rangle\de s+\dot x_1(t)x_1(t)\int_0^L \scp{K_\xi(s,t)e_1}{e_2}\de s= \\
& -\int_0^L\langle K_\xi(s,t)\dot\xi(s,t),J\xi(s,t)\rangle\,\de s-x_1(t)\int_0^L\langle K_\xi(s,t)\dot\xi(s,t),e_2\rangle\,\de s.
\end{split}
\end{equation}
To solve simultaneously these equations for the unknown $\dot x_1(t)$, we will show that the second components of the integrals in \eqref{139} are zero, that the first component of the integral in the left-hand side in \eqref{139} is non zero, and  that all integrals in \eqref{921} are zero.

\par Let us start from the second components of \eqref{139}. 
For the left-hand side, we have
\begin{equation*}\label{ts10}
\begin{split}
\scp{K_\xi(s,t)e_1}{e_2}= & \scp{K_0(s,u_1(t),u_2(t))e_1}{e_2} \\
= & (C_\tau-C_\nu) \cos(u_1(t)\vartheta_0(s-u_2(t)))\sin(u_1(t)\vartheta_0(s-u_2(t))) 
\end{split} 
\end{equation*}
so it suffices to show that for every $u_1,u_2$
\begin{equation}\label{ts11}
\int_0^L \cos(u_1\vartheta_0(s-u_2))\sin(u_1\vartheta_0(s-u_2))\,\de s=0.
\end{equation}
By \eqref{ts12} and by changing variables $\sigma=s-u_2$, \eqref{ts11} becomes
\begin{equation*}\label{ts13}
\int_{-\ell}^{\ell} \cos(u_1\vartheta_0(\sigma))\sin(u_1\vartheta_0(\sigma))\,\de\sigma=0,
\end{equation*}
which holds true since the integrand is an odd function in $[-\ell,\ell]$, thanks to \eqref{0002a}.

\par For the second component of the right-hand side of \eqref{139}, we have
\begin{equation*}\label{ts14}
\begin{split}
\langle K_\xi(s,t)&\dot\xi(s,t),e_2\rangle= C_\tau \dot u_1(t)\sin(u_1(t)\vartheta_0(s-u_2(t)))\cdot \\
& \cdot\left[-\cos(u_1(t)\vartheta_0(s-u_2(t)))\int_0^s\sin(u_1(t)\vartheta_0(\sigma-u_2(t)))\vartheta_0(\sigma-u_2(t))\,\de\sigma\right. \\
& \quad\left. +\sin(u_1(t)\vartheta_0(s-u_2(t)))\int_0^s \cos(u_1(t)\vartheta_0(\sigma-u_2(t)))\vartheta_0(\sigma-u_2(t))\,\de\sigma\right] \\
& - C_\nu \dot u_2(t) \cos(u_1(t)\vartheta_0(s-u_2(t)))\sin(u_1(t)\vartheta_0(s-u_2(t))).
\end{split}
\end{equation*}
Therefore, we need to prove that for any $u_1,u_2$
\begin{equation}\label{ts15}
\int_0^L \sin(2u_1\vartheta_0(s-u_2))\left[\int_0^s \sin(u_1\vartheta_0(\sigma-u_2))\vartheta_0(\sigma-u_2)\,\de\sigma\right]\de s=0,
\end{equation}
\begin{equation}\label{ts16}
\int_0^L \sin^2(u_1\vartheta_0(s-u_2))\left[\int_0^s \cos(u_1\vartheta_0(\sigma-u_2))\vartheta_0(\sigma-u_2)\,\de\sigma\right]\de s=0.
\end{equation}
Again by \eqref{ts12} and by changing variables as before, \eqref{ts15} and \eqref{ts16} reduce to
\begin{equation}\label{ts17}
\int_{-\ell}^\ell \sin(2u_1\vartheta_0(s))\left[\int_{-\ell}^s \sin(u_1\vartheta_0(\sigma))\vartheta_0(\sigma)\,\de\sigma\right]\de s=0,
\end{equation}
\begin{equation}\label{ts18}
\int_{-\ell}^\ell \sin^2(u_1\vartheta_0(s))\left[\int_{-\ell}^s \cos(u_1\vartheta_0(\sigma))\vartheta_0(\sigma)\,\de\sigma\right]\de s=0.
\end{equation}
Since $\sin(2u_1\vartheta_0(s))$ is odd in $[-\ell,\ell]$ by \eqref{0002a}, equation \eqref{ts17} is equivalent to
\begin{equation}\label{ts19}
\int_{-\ell}^\ell \sin(2u_1\vartheta_0(s))\left[\int_{0}^s \sin(u_1\vartheta_0(\sigma))\vartheta_0(\sigma)\,\de \sigma\right]\de s=0.
\end{equation}
By \eqref{0002a} and by Lemma \ref{220} the function $s\mapsto\int_0^s \sin(u_1\vartheta_0(\sigma))\vartheta_0(\sigma)\,\de \sigma$ is odd in $[-\ell,\ell]$, therefore \eqref{ts19} is equivalent to
\begin{equation}\label{ts20}
\int_{0}^\ell \sin(2u_1\vartheta_0(s))\left[\int_{0}^s \sin(u_1\vartheta_0(\sigma))\vartheta_0(\sigma)\,\de \sigma\right]\de s=0,
\end{equation}
since its integrand is even in $[-\ell,\ell]$. 
Since $\sin(2u_1\vartheta_0(s))$ is even in $[0,\ell]$ by \eqref{0002b}, we have
\begin{equation*}\label{ts22}
\int_{0}^\ell \sin(2u_1\vartheta_0(s))\,\de s=2\int_0^{\ell/2} \sin(2u_1\vartheta_0(s))\,\de s=0,
\end{equation*}
where the last equality follows from the fact that $\sin(2u_1\vartheta_0(s))$ is odd in $[0,\ell/2]$ by \eqref{0002c}. 
This equality implies that \eqref{ts20} reduces to
\begin{equation}\label{ts21}
\int_{0}^\ell \sin(2u_1\vartheta_0(s))\left[\int_{\ell/2}^s \sin(u_1\vartheta_0(\sigma))\vartheta_0(\sigma)\,\de \sigma\right]\de s=0.
\end{equation}
By \eqref{0002b}, $\sin(2u_1\vartheta_0(s))$, and $\sin(u_1\vartheta_0(s))\vartheta_0(s)$ are even in $[0,\ell]$, hence the function $s\mapsto\int_{\ell/2}^s \sin(u_1\vartheta_0(\sigma))\vartheta_0(\sigma)\,\de \sigma$ is odd in $[0,\ell]$ by Lemma \ref{220}. This implies that \eqref{ts21} holds, since its integrand is odd in $[0,\ell]$. This concludes the proof of \eqref{ts17}.

\par  To prove \eqref{ts18} we notice that in the function $s\mapsto\cos(u_1\vartheta_0(s))\vartheta_0(s)$ is even in $[-\ell,0]$ by \eqref{0002b}. Hence,
\begin{equation}\label{ts23}
\int_{-\ell}^0 \cos(u_1\vartheta_0(\sigma))\vartheta_0(\sigma)\,\de \sigma=2\int_{-\ell/2}^{0} \cos(u_1\vartheta_0(\sigma))\vartheta_0(\sigma)\,\de \sigma=0,
\end{equation}
where the last equality follows from the fact that the function $s\mapsto\cos(u_1\vartheta_0(s))\vartheta_0(s)$ is odd in $[-\ell/2,0]$ by \eqref{0002c}. This implies that \eqref{ts18} is equivalent to
\begin{equation}\label{ts24}
\int_{-\ell}^\ell \sin^2(u_1\vartheta_0(s))\left[\int_{0}^s \cos(u_1\vartheta_0(\sigma))\vartheta_0(\sigma)\,\de \sigma \right]\de s=0.
\end{equation}
By \eqref{0002a}, the function $\cos(u_1\vartheta_0(\sigma))\vartheta_0(\sigma)$ is odd in $[-\ell,\ell]$, and therefore the function $s\mapsto\int_0^s \cos(u_1\vartheta_0(\sigma))\vartheta_0(\sigma)\,\de \sigma$ is even in $[-\ell,\ell]$ by Lemma \ref{220}. Since also $\sin^2(u_1\vartheta_0(s))$ is even in $[-\ell,\ell]$ by \eqref{0002a}, \eqref{ts24} is equivalent to
\begin{equation*}
\int_{0}^\ell \sin^2(u_1\vartheta_0(s))\left[\int_{0}^s \cos(u_1\vartheta_0(\sigma))\vartheta_0(\sigma)\,\de \sigma \right]\de s=0,
\end{equation*}
which, by \eqref{ts23}, is equivalent to
\begin{equation}\label{ts25}
\int_{0}^\ell \sin^2(u_1\vartheta_0(s))\left[\int_{\ell/2}^s \cos(u_1\vartheta_0(\sigma))\vartheta_0(\sigma)\,\de \sigma \right]\de s=0.
\end{equation}
By \eqref{0002b}, the function $\cos(u_1\vartheta_0(\sigma))\vartheta_0(\sigma)$ is even in $[0,\ell]$, and therefore the function $s\mapsto\int_{\ell/2}^s \cos(u_1\vartheta_0(\sigma))\vartheta_0(\sigma)\,\de \sigma$ is odd in $[0,\ell]$ by Lemma \ref{220}. Since also $\sin^2(u_1\vartheta_0(s))$ is even in $[0,\ell]$ by \eqref{0002b}, \eqref{ts25} holds because its integrand is odd in $[0,\ell]$. This concludes the proof of \eqref{ts18}.
Therefore, we have proved that the second component in \eqref{139} vanishes.

\par By the results just proved, \eqref{921} reduces to
\begin{equation}\label{ts26}
\dot x_1(t)\int_0^L \langle K_\xi(s,t)e_1,J\xi(s,t)\rangle\de s= -\int_0^L\langle K_\xi(s,t)\dot\xi(s,t),J\xi(s,t)\rangle\,\de s
\end{equation}
To prove that the left-hand side is zero, by \eqref{r105} and \eqref{r106} it is enough to show that
\begin{equation*}\label{ts27}
\int_0^L \scp{\xi_0'(s,t)}{e_1}\scp{\xi_0'(s,t)}{J\xi_0(s,t)}\,\de s=\int_0^L \scp{J\xi_0'(s,t)}{e_1}\scp{J\xi_0'(s,t)}{J\xi_0(s,t)}\,\de s=0,
\end{equation*}
which in turn is valid if we prove that, for any $u_1,u_2$ (after recalling \eqref{ts12} and performing the usual change of variables $s-u_2\to s$)
\begin{equation*}\label{ts28}
\int_{-\ell}^\ell \sin(2u_1\vartheta_0(s))\left[\int_{-\ell}^s \cos(u_1\vartheta_0(\sigma))\,\de \sigma\right]\de s=0,
\end{equation*}
\begin{equation*}\label{ts29}
\int_{-\ell}^\ell \sin^2(u_1\vartheta_0(s))\left[\int_{-\ell}^s \sin(u_1\vartheta_0(\sigma))\,\de \sigma \right]\de s=0,
\end{equation*}
\begin{equation*}\label{ts30}
\int_{-\ell}^\ell \cos^2(u_1\vartheta_0(s))\left[\int_{-\ell}^s \sin(u_1\vartheta_0(\sigma))\,\de \sigma \right]\de s=0.
\end{equation*}
To prove these equalities we can argue as in \eqref{ts17} and \eqref{ts18}.

\par To prove that the right-hand side of \eqref{ts26} is zero, besides the equalities already proved, we have to show that \begin{equation*}
\int_{-\ell}^\ell \cos(u_1\vartheta_0(s))\left[\int_{-\ell}^s \sin(u_1\vartheta_0(\sigma))\,\de \sigma\right]\de s=0,
\end{equation*}
\begin{equation*}
\int_{-\ell}^\ell \sin(u_1\vartheta_0(s))\left[\int_{-\ell}^s \cos(u_1\vartheta_0(\sigma))\,\de \sigma\right]\de s=0,
\end{equation*}
\begin{equation*}
\int_{-\ell}^\ell \cos^2(u_1\vartheta_0(s))\left[\int_{-\ell}^s \sin(u_1\vartheta_0(\sigma))\vartheta_0(\sigma)\,\de \sigma\right]\!\! \left[\int_{-\ell}^s \sin(u_1\vartheta_0(\sigma))\,\de \sigma\right]\de s=0,
\end{equation*}
\begin{equation*}
\int_{-\ell}^\ell \sin^2(u_1\vartheta_0(s))\left[\int_{-\ell}^s \sin(u_1\vartheta_0(\sigma))\vartheta_0(\sigma)\,\de \sigma\right]\!\! \left[\int_{-\ell}^s \sin(u_1\vartheta_0(\sigma))\,\de \sigma\right]\de s=0,
\end{equation*}
\begin{equation*}
\int_{-\ell}^\ell \!\!\sin(2u_1\vartheta_0(s)) \!\left[\int_{-\ell}^s \!\!\sin(u_1\vartheta_0(\sigma))\vartheta_0(\sigma)\,\de \sigma\right] \!\!\left[\int_{-\ell}^s \!\!\cos(u_1\vartheta_0(\sigma))\,\de \sigma\right]\de s=0,
\end{equation*}
\begin{equation*}
\int_{-\ell}^\ell \!\!\sin(2u_1\vartheta_0(s)) \!\left[\int_{-\ell}^s \!\!\cos(u_1\vartheta_0(\sigma))\vartheta_0(\sigma)\,\de \sigma\right] \!\!\left[\int_{-\ell}^s \!\!\sin(u_1\vartheta_0(\sigma))\,\de \sigma\right]\de s=0,
\end{equation*}
\begin{equation*}
\int_{-\ell}^\ell \sin^2(u_1\vartheta_0(s))\left[\int_{-\ell}^s \cos(u_1\vartheta_0(\sigma))\vartheta_0(\sigma)\,\de \sigma\right]\!\! \left[\int_{-\ell}^s \cos(u_1\vartheta_0(\sigma))\,\de \sigma\right]\de s=0,
\end{equation*}
\begin{equation*}
\int_{-\ell}^\ell \cos^2(u_1\vartheta_0(s))\left[\int_{-\ell}^s \cos(u_1\vartheta_0(\sigma))\vartheta_0(\sigma)\,\de \sigma\right]\!\! \left[\int_{-\ell}^s \cos(u_1\vartheta_0(\sigma))\,\de \sigma\right]\de s=0.
\end{equation*}
This can be done as in the previous proofs, using the symmetry assumptions \eqref{0002} together with Lemma \ref{220}.

\par We still need to verify that the first component in the left-hand side of \eqref{139} does not vanish.
Indeed, by using \eqref{r103}, \eqref{r105}, and \eqref{r106},
\begin{equation*}\label{ts38}
\begin{split}
\int_0^L \scp{K_\xi(s,t)e_1}{e_1}\de s = & \int_{-\ell}^\ell [C_\tau\cos^2(u_1\vartheta_0(s))+C_\nu\sin^2(u_1\vartheta_0(s))]\de s,
\end{split}
\end{equation*}
which is clearly greater than zero.
Therefore, \eqref{139} can be written in the following way
\begin{equation}\label{ts39}
\dot x_1(t)=\frac{F^\sh(t)}{b(u_1(t))}=a_1(u_1(t))\dot u_1(t)+a_2(u_1(t))\dot u_2(t),
\end{equation}
where we have set
\begin{subequations}\label{ts40}
\begin{eqnarray}
a_i(u_1) &:=& -\frac{b_i(u_1)}{b(u_1)}, \label{ts40a} \\
b(u_1) &:=& \int_0^L\scp{K_0(s,u_1,u_2)e_1}{e_1}\de s, \label{ts40b} \\
b_i(u_1) &:=& \int_0^L \scp{K_0(s,u_1,u_2)\xi_{0,i}(s,u_1,u_2)}{e_1}\de s, \label{ts40c}
\end{eqnarray}
\end{subequations}
since the right-hand sides in \eqref{ts40} are in fact independent of $u_2$.
Indeed, an easy computation recalling \eqref{r105}, \eqref{r109}, and \eqref{r110} leads to the following expressions
\begin{subequations}\label{e003}
\begin{eqnarray}
b(u_1)& =& 2C_\tau \ell +(C_\nu-C_\tau)\int_{-\ell}^\ell \sin^2(u_1\vartheta_0(s))\,\ds, \label{e003a} \\
b_1(u_1)& =& -C_\tau\int_{-\ell}^\ell \left[\int_{-\ell}^s\vartheta_0(\sigma)\sin(u_1\vartheta_0(\sigma))\dsigma\right]\ds \nonumber\\
&& -(C_\nu-C_\tau)\int_{-\ell}^\ell\sin^2(u_1\vartheta_0(s)) \left[\int_{-\ell}^s\vartheta_0(\sigma)\sin(u_1\vartheta_0(\sigma))\dsigma\right]\ds \label{e003b} \\
&& +\frac{C_\tau-C_\nu}{2}\int_{-\ell}^\ell\sin(2u_1\vartheta_0(s)) \left[\int_{-\ell}^s\vartheta(\sigma)\cos(u_1\vartheta_0(\sigma))\dsigma\right]\ds \nonumber \\
b_2(u_1)&=& 2C_\tau \ell -C_\tau \!\! \int_{-\ell}^\ell \!\! \cos(u_1\vartheta_0(s))\ds{+} (C_\nu-C_\tau) \!\! \int_{-\ell}^\ell \!\! \sin^2(u_1\vartheta_0(s))\ds. \label{e003c}
\end{eqnarray}
\end{subequations}

\par Notice that from \eqref{139} and \eqref{ts39}
\begin{equation}\label{ts50}
F^\sh(t)=[b_1(u_1(t))\dot u_1(t)+b_2(u_1(t))\dot u_2(t)]e_1.
\end{equation}

\par Assume that $u_1(0)=0$. Since $x_1(0)=0$ and $\xi(0,0)=0$ (recall that we have assumed $x(0)=0$ and $\xi(0,t)=0$ for all $t\in[0,1]$), the initial condition $\chi(\cdot,0)=\Sigma_\init(\cdot)$ is satisfied. 

\par Assume also that $u_1(1)=0$. Since $\xi(0,1)=0$, the final condition $\chi(\cdot,1)=\Sigma_\fin(\cdot)$ is satisfied provided $x_1(1)=a$. Therefore we have to show that we can choose the Lipschitz controls $u_1:[0,1]\to[-1,1]$ and $u_2:[0,1]\to[\ell,L-\ell]$ in such a way that the corresponding solution of \eqref{ts39} satisfying the initial conditions $x_1(0)=0$ satisfies also the final condition $x_1(1)=a$.

\par Equation \eqref{ts39} shows that $x_1(1)$ is the integral of the differential form $a_1(u_1)\de u_1+a_2(u_1)\de u_2$ along the oriented curve $t\mapsto\vec{u_1(t)}{u_2(t)}$.

\par We shall prove that there exists $\delta\in(0,1)$ such that
\begin{equation}\label{ts5}
\frac{\de a_2}{\de u_1}(0)=0\qquad\text{and}\qquad
\frac{\de a_2}{\de u_1}(u_1)<0,\qquad\text{for $0<u_1<\delta$.}
\end{equation}

For every $\gamma\leq\delta$, let us consider the rectangle $\cR_\gamma:=[0,\gamma]{\times}[\ell,L-\ell]$. 
If $(u_1^\gamma(t),u_2^\gamma(t))$ is a clockwise parametrization of $\partial\cR_\gamma$ with $u_1^\gamma(0)=u_1^\gamma(1)=0$, then the corresponding solutions $x_1^\gamma$ to \eqref{ts39} with initial condition $x_1^\gamma(0)=0$ satisfy the condition 
\begin{equation*}\label{ts6}
0<x_1^\gamma(1)\leq x_1^\delta(1).
\end{equation*}
If $0<a\leq x_1^\delta(1)$, by continuity there exists $\gamma\in(0,\delta]$ such that $x_1^\gamma(1)=a$.
If $a>x_1^\delta(1)$, then we can write $a=nb$, with $n\in\mathbb{N}$ and $b\in(0,x_1^\delta(1)]$, and fix $\gamma$ such that $x_1^\gamma(1)=b$. 
We then extend $u_1$ and $u_2$ by $1$-periodicity, and we achieve the equality $x_1(1)=a$ by choosing as control $(u_1(t),u_2(t))=(u_1^\gamma(nt),u_2^\gamma(nt))$.

\par We remark that this elementary argument, based on the non-integrability of the differential form $a_1(u_1)\de u_1+a_2(u_1)\de u_2$, is a particular case of a well know result in Geometric Control Theory, namely Chow's Theorem, see, e.g., \cite[Theorem 3.18]{Coron}.

\par We now prove \eqref{ts5}. By \eqref{e003a} and \eqref{e003c} it is easy to compute 
\begin{equation}\label{ts42}
\begin{split}
\frac{\de b}{\de u_1}(u_1) = & (C_\nu-C_\tau)\int_{-\ell}^\ell \sin(2u_1\vartheta_0(s))\vartheta_0(s)\,\ds, \\
\frac{\de b_2}{\de u_1}(u_1)= & C_\tau\int_{-\ell}^\ell\sin(u_1\vartheta_0(s))\vartheta_0(s)\ds+ (C_\nu-C_\tau)\int_{-\ell}^\ell\sin(2u_1\vartheta_0(s))\vartheta_0(s)\ds,
\end{split}
\end{equation}
\begin{equation}\label{ts43}
b(0)=2C_\tau\ell,\qquad b_2(0)=0,\qquad 
\frac{\de b}{\de u_1}(0)=0,\qquad \frac{\de b_2}{\de u_1}(0)=0,
\end{equation}
which implies the first equality in \eqref{ts5}, since
\begin{equation}\label{ts41}
b^2(u_1)\frac{\de a_2}{\de u_1}(u_1)=-\frac{\de b_2(u_1)}{\de u_1}b(u_1)+b_2(u_1)\frac{\de b(u_1)}{\de u_1}.
\end{equation}
To establish the result on the sign of $\frac{\de a_2}{\de u_1}(u_1)$, it is enough to study the derivative of the right-hand side of \eqref{ts41}, which is equal to 
\begin{equation}\label{ts44}
-\frac{\de^2 b_2}{\de u_1^2}(u_1)b(u_1)+b_2(u_1)\frac{\de^2 b}{\de u_1^2}(u_1).
\end{equation}
From \eqref{ts42} we obtain
\begin{equation*}\label{ts45}
\begin{split}
\frac{\de^2 b_2}{\de u_1^2}(u_1)= & C_\tau\int_{-\ell}^\ell\cos(u_1\vartheta_0(s))\vartheta_0^2(s)\ds+ 2(C_\nu-C_\tau)\int_{-\ell}^\ell\cos(2u_1\vartheta_0(s))\vartheta_0^2(s)\ds,
\end{split}
\end{equation*}
and therefore
\begin{equation*}\label{ts46}
\frac{\de^2 b_2}{\de u_1^2}(0)=  (2C_\nu-C_\tau)\int_{-\ell}^\ell \vartheta_0^2(s)\ds.
\end{equation*}
Evaluating \eqref{ts44} at $u_1=0$ gives (recall the second equality in \eqref{ts43})
\begin{equation*}\label{ts47}
\sign\left(\frac{\de^2 a_2}{\de u_1^2}(0)\right)=\sign\left(-2C_\tau \ell(2C_\nu-C_\tau)\int_{-\ell}^\ell \vartheta_0^2(s)\ds\right)<0,
\end{equation*}
which proves the second part of claim \eqref{ts5}.

\subsection{Rotation}\label{rotation}
\par In this subsection we describe how to rotate a straight swimmer about its center in the time interval $[0,1]$.  This will be obtained in three steps. In the first one we deform symmetrically the initial segment into the shape in Fig.\@ \ref{fR1}, with two parallel straight terminal parts; by symmetry the deformation process will produce a rotation of an angle $\varphi_0$ (that we will not estimate) about the midpoint. In the second step we propagate bumps on the rectilinear parts as described below in order to achieve a rotation of a prescribed angle $\varphi$. In the third step, we straighten back the now rotated configuration in Fig.\@ \ref{fR1} into a straight one by reverting the process in step one: this will produce a rotation of angle $-\varphi_0$ about the midpoint, so that at the end of the process the segment will be rotated by the angle $\varphi$.
\begin{figure}[htbp]
\centering
\includegraphics[scale=.5]{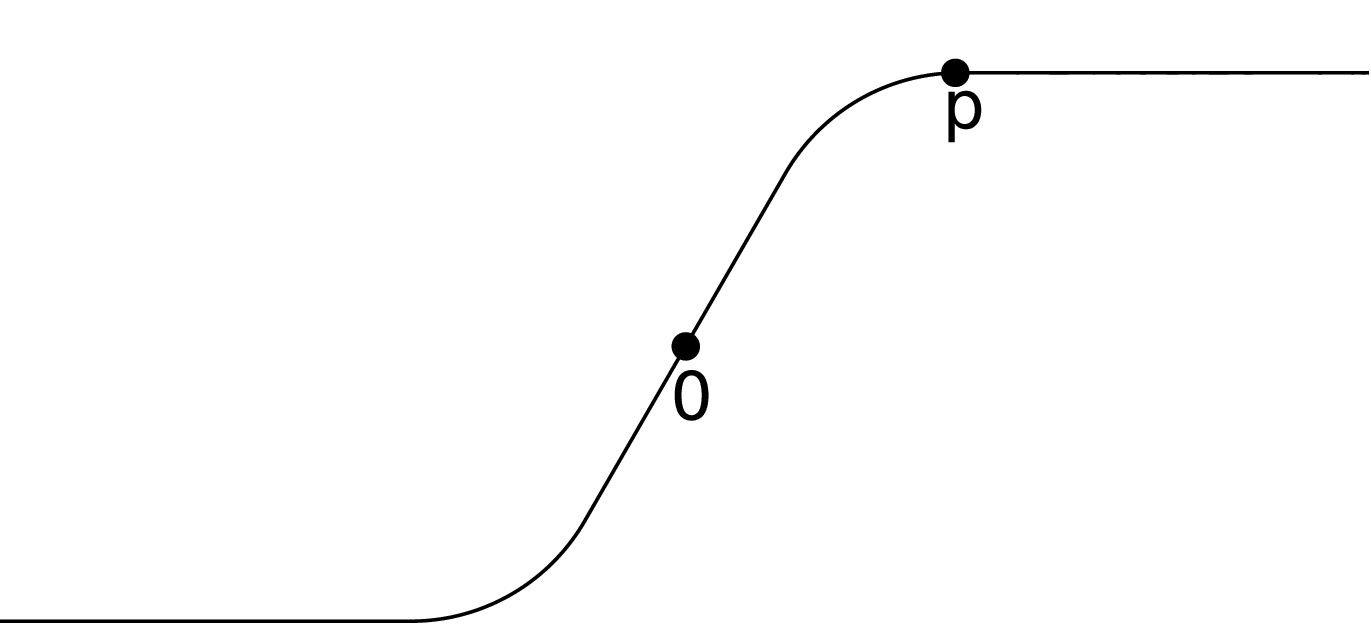}
\caption{Rotation: the configurations $\xi(s,1/3)=\xi(s,2/3)$.}
\label{fR1}
\end{figure}

\par Without loss of generality, in this section it is convenient to assume that the length of the swimmer is $2L$ and to parametrize all curves in the interval $[-L,L]$. We take $\Sigma_\init(s)=se_1$ and $\Sigma_\fin(s)=se(\varphi)$, for $s\in[-L,L]$, where $e(\varphi)=(\cos\varphi,\sin\varphi)$, $\varphi$ being the angle of rotation. 
As before, the motion will be first described through a function $\xi\in X_1$ satisfying the two disks condition with radius $\rho$. 
Then, we will consider the function $\chi(s,t)$ defined by \eqref{1001}, where $x(t)$ and $R(t)$ satisfy the equation of motion \eqref{1027}. The initial and final conditions on $\chi$ are
\begin{equation}\label{711}
\chi(s,0)=se_1\,,\qquad \chi(s,1)=se(\varphi),
\end{equation}
for all $s\in[-L,L]$. 

\par We also assume that $\xi(s,0)=se_1$ and that $\xi$ satisfies
\begin{equation}\label{707}
\xi(s,t)=-\xi(-s,t)
\end{equation}
for every $s\in[-L,L]$ and for every $t\in[0,1]$. It follows that
\begin{equation}\label{170}
\xi'(s,t)=\xi'(-s,t),\qquad\qquad\dot\xi(s,t)=-\dot\xi(-s,t),
\end{equation}
which implies that the force density $f(s,t)$ in \eqref{104} is odd with respect to $s$, so that $F(t)=\int_{-L}^{L} f(s,t)\,\de s=0$ for all $t\in[0,1]$.

\par The quantities introduced in \eqref{709a}, \eqref{006}, and \eqref{1028} are now defined by integration over the interval $[-L,L]$. The symmetry properties \eqref{707} and \eqref{170} imply also that the vector $b(t)$ introduced in \eqref{709} and the vector $F^\sh(t)$ defined in \eqref{006} vanish. As a consequence, the vectors $\bar b(t)$ and $v(t)$ introduced in \eqref{1028} are zero and $\bar c(t)=1/c(t)$. Therefore, the equations of motion \eqref{1027} read $\dot x(t)=0$ and 
\begin{equation}\label{712}
\dot\theta(t)=\omega(t):=\frac{M^\sh(t)}{c(t)},
\end{equation}
where $\theta(t)$ is the angle of the rotation $R(t)$.
Together with the initial conditions at time $t=0$, this implies that $x(t)=0$ for every $t\in[0,1]$, $\theta(0)=0$, and
\begin{equation*}
\chi(s,t)=R(t)\xi(s,t).
\end{equation*}
Therefore, the final condition in \eqref{711} is equivalent to
\begin{equation*}
\theta(1)\equiv\varphi \mod 2\pi.
\end{equation*}

\par The first step will take place in the time interval $[0,1/3]$. 
The curve $\xi(\cdot,1/3)$ is the one represented in Fig.\@ \ref{fR1}. The main feature of this curve, besides being odd, is that $\xi(s,1/3)=p+(s-L/2)e_1$, for $s\in[L/2,L]$, where $p=\xi(L/2,1/3)$ and  $p_2>0$. 
The angle $\varphi_0$ mentioned at the beginning of the section is then defined by $\varphi_0=\theta(1/3)$.

\par The second step will take place in the time interval $[1/3,2/3]$. 
As in the case of pure translations, the overall rotation of angle $\varphi$ will be achieved by iterating the cyclic motions described below. 

\par During each cycle, we deform the rectilinear parts of the swimmer with the same bumps we used for the translation, see Fig.\@ \ref{fR2}. 
\begin{figure}[htbp]
\centering
\includegraphics[scale=.3]{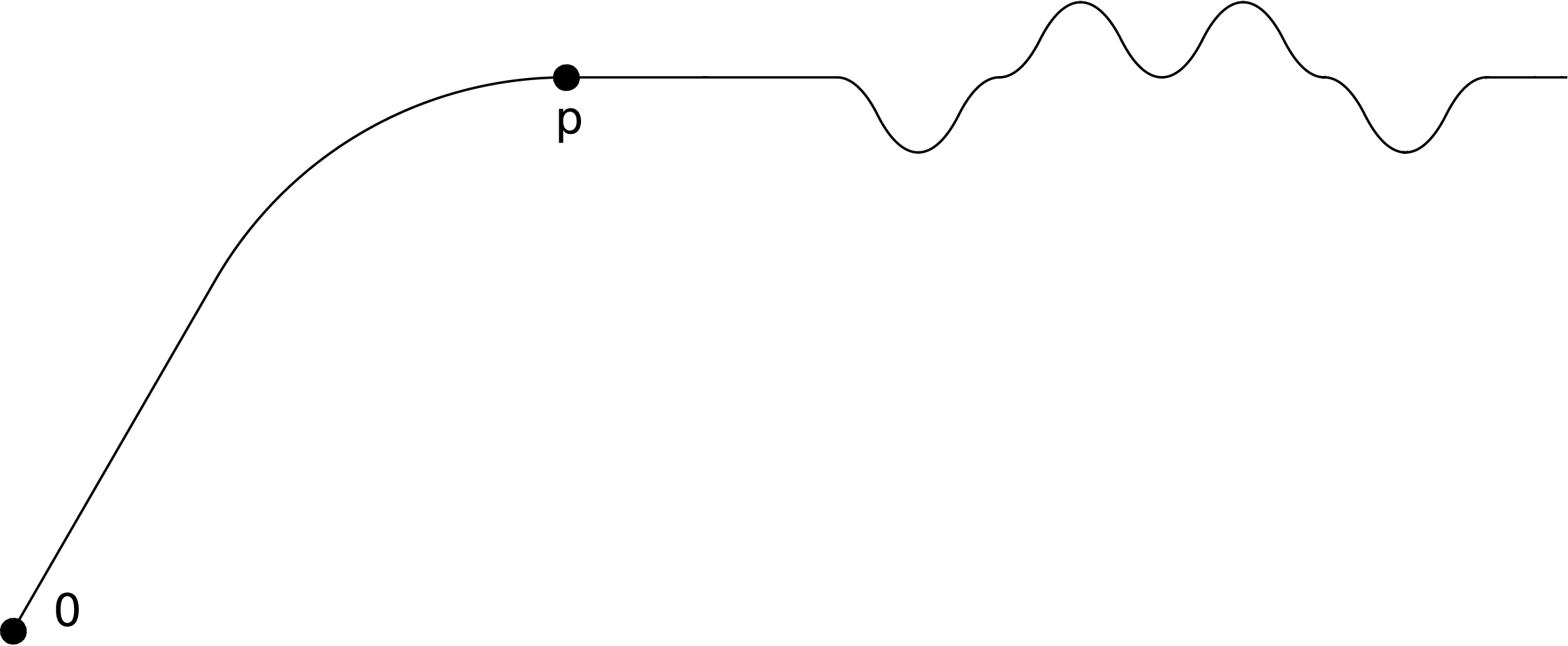}
\caption{Rotation: the right half of the swimmer is shown, $s\in[0,L]$.}
\label{fR2}
\end{figure}
They will be created at the ends of the swimmer, will travel towards its center, and will be destroyed before entering the curvilinear part.

\par To describe the geometry of the bumps, we use the angle function $\vartheta(s,t)$ considered in \eqref{120}, where now the Lipschitz controls $u_1$ and $u_2$ take values in $[-1,1]$ and $[L/2+\ell,L-\ell]$, with $\ell\in(0,L/4)$.
The corresponding function $\xi(s,t)$ is defined by $\xi(s,t)=\xi(s,1/3)$ for $s\in[0,L/2]$ and $t\in[1/3,2/3]$, and by $\xi(s,t)=\xi_0(s,u_1(t),u_2(t))$, where
\begin{equation*}\label{ts70}
\xi_0(s,u_1,u_2):=p+\int_{L/2}^s \vec{\cos(u_1\vartheta_0(\sigma-u_2))}{\sin(u_1\vartheta_0(\sigma-u_2))}\de \sigma.
\end{equation*}
\par Using \eqref{ts50}, the force $F^\sh(t)$ generated during the translation of the bumps is given by 
\begin{equation*}\label{997}
F_\rig^\sh(t)=
[b_1(u_1(t))\dot u_1(t)+b_2(u_1(t))\dot u_2(t)]
e_1,\qquad\text{for every $t\in[1/3,2/3]$.}
\end{equation*}
By the symmetries introduced in \eqref{707}, the global torque defined in \eqref{006} can be computed as
\begin{equation*}\label{999}
\begin{split}
M^\sh(t)= & -2\int_{L/2}^L \langle J(\xi(s,t)-p),K_\xi(s,t)\dot\xi(s,t)\rangle\,\de s+2\scp{Jp}{F_\rig^\sh(t)} \\ 
= & -2p_2[b_1(u_1(t))\dot u_1(t)+b_2(u_1(t))\dot u_2(t)]
\end{split}
\end{equation*}
where the integral, which represents the torque with respect to the point $p$, vanishes, as explained in Section \ref{translation}.

\par To compute $\dot\theta(t)$ in \eqref{712}, it is convenient to write
\begin{equation*}\label{ts71}
c(t)=b^*(u_1(t)),
\end{equation*}
where, by \eqref{709b}, \eqref{707}, and \eqref{170}
\begin{equation*}\label{ts72}
b^*(u_1):=2C_\tau\int_{L/2}^L \scp{\xi_0'(s,u_1,u_2)}{J\xi_0(s,u_1,u_2)}^2\de s+2C_\nu\int_{L/2}^L \scp{\xi_0'(s,u_1,u_2)}{\xi_0(s,u_1,u_2)}^2\de s;
\end{equation*}
it is easy to see that the right-hand side above is independent of $u_2$.
Equation \eqref{712} becomes now
\begin{equation*}\label{ts73}
\dot\theta(t)=a_1^*(u_1(t))\dot u_1(t)+ a_2^*(u_1(t))\dot u_2(t),
\end{equation*}
where
\begin{equation*}\label{ts74}
a_i^*(u_1):=\frac{b_i(u_1)}{b^*(u_1)}.
\end{equation*}

\par Now, with the same strategy used for proving
\eqref{ts5}, we can show that there exists $\delta\in(0,1)$ such that
\begin{equation*}\label{ts75}
\frac{\de a_2^*}{\de u_1}(0)=0\qquad\text{and}\qquad
\frac{\de a_2^*}{\de u_1}(u_1)<0,\qquad\text{for $0<u_1<\delta$.}
\end{equation*}
We can now conclude step two as in Section \ref{translation} and we get $\theta(2/3)=\theta(1/3)+\varphi=\varphi_0+\varphi$.

\par The third step will take place in the time interval $[2/3,1]$. We now define $\xi(s,t)=\xi(s,1-t)$ for $s\in[-L,L]$ and for $t\in[2/3,1]$. Since this motion is the same as in step one with time reversed, the rotation angle $\theta(1)-\theta(2/3)$ will be equal to $-\varphi_0$\,, hence $\theta(1)=\varphi$.
\end{proof}

\section{Existence of an optimal swimming strategy}

\par 
In this section we prove Theorem \ref{010} about the existence of an energetically optimal swimming strategy. The result is achieved by proving that a minimum problem for the power expended \eqref{007} has a solution.

Let us recall the definition of power expended:
\begin{equation}\label{200}
\cP(\chi):=\int_0^L\!\!\int_0^T \scp{-f(s,t)}{\dot\chi(s,t)}\,\de s\de t=\int_0^L\!\!\int_0^T \scp{K_\chi(s,t)\dot\chi(s,t)}{\dot\chi(s,t)}\,\de s\de t.
\end{equation}
Up to a change of coordinates, it is possible to represent $K_\chi(s,t)$ in diagonal form, with entries $C_\tau$ and $C_\nu$\,. Since $0<C_\tau<C_\nu$\,, the matrix $K_\chi(s,t)$ is positive definite and its lower eigenvalue is $C_\tau$. It follows that
\begin{equation}\label{1008}
\cP(\chi)
\geq C_\tau \int_0^L\!\!\int_0^T \norm{\dot\chi(s,t)}^2\de s\de t.
\end{equation}

\par For every $\rho>0$ let $X_1^\rho$ be the set of all functions $\chi \in X_1$ (see \eqref{836}) such that 
for every $t\in[0,T]$ the curve $\chi(\cdot,t)$ satisfies the external disks condition with radius $\rho$ (see Definition \ref{755}).
\begin{theorem}\label{010} Let $\rho>0$ and let $\chi_\init$\,, $\chi_\fin\in H^2(0,L;\R2)$, with  $|\chi_\init'(s)|=|\chi_\fin'(s)|=1$ for every $s\in[0,L]$. Assume that $\chi_\init$ and $\chi_\fin$ satisfy the two disks condition with radius $\rho$ (see Definition \ref{755}). Then the minimum problem
\begin{equation}\label{180}
\min\{\cP(\chi):\chi\in X_1^\rho, \text{ \emph{\eqref{105} holds}},\ \chi(\cdot,0)=\chi_\init(\cdot),\ \chi(\cdot,T)=\chi_\fin(\cdot)\}
\end{equation}
has a solution.
\end{theorem}
\begin{proof} 
We first observe that the set of motions $\chi$ on which we are mimimizing is nonempty by Theorem \ref{300}.
Let us consider a minimizing sequence $(\chi_k)_k$. By \eqref{1008} there exists a
constant $M<+\infty$ such that,
\begin{equation}\label{1801}
\int_0^L\!\!\int_0^T \norm{\dot\chi_k(s,t)}^2\de s\de t\leq M
\end{equation}
for every $k$.
Notice that the external disks condition with radius $\rho$ gives the estimate
\begin{equation}\label{1802}
\norm{\chi_k^{\prime\prime}(s,t)}\leq 1/\rho
\end{equation}
for every $t\in [0,T]$ and for a.e.\ $s\in [0,L]$.

We now show that $\chi_k(\cdot,t)$ is bounded in $L^2(0,L)$ uniformly with respect to $k$ and $t$.
Since $\chi_k(\cdot,0)=\chi_\init(\cdot)$, for every $s\in [0,L]$ we have
$$
\norm{\chi_k(s,t)}^2\leq2\norm{\chi_\init(s)}^2+2T\int_0^T \norm{\dot\chi_k(s,t)}^2\de t.
$$
From this inequality and from \eqref{1801} we get
\begin{equation}\label{1803}
\sup_{t\in[0,T]} \norma{\chi_k(\cdot,t)}_{\Lp2{}(0,L)}^2\leq 2\norma{\chi_\init}_{\Lp2{}(0,L)}^2+2TM.
\end{equation}
By an elementary interpolation inequality we deduce from \eqref{1802} and \eqref{1803} that
\begin{equation}\label{1804}
\sup_{t\in[0,T]} \norma{\chi_k(\cdot,t)}_{\Hp2{}(0,L)}\leq C,
\end{equation}
for a suitable constant $C<+\infty$ independent of $k$.

By \eqref{1801} and \eqref{1803} the sequence  $(\chi_k)_k$ is bounded in $\Hp1{}(0,T;\Lp2{}(0,L))$. Therefore there exist a subsequence, not relabeled, and a function $\chi:[0,L]{\times}[0,T]\to\R2 $
such that
\begin{eqnarray}
&\label{1811}
\chi\in\Hp1{}(0,T;\Lp2{}(0,L)),
\\
&\label{181}
\chi_k\wto\chi\quad\text{weakly in $\Hp1{}(0,T;\Lp2{}(0,L))$.}
\end{eqnarray}
Since $\Hp1{}(0,T;\Lp2{}(0,L))$ is continuously embedded into $\Cl0{}([0,T];\Lp2{}(0,L))$ and for every
$t\in[0,T]$ the function
$\xi\mapsto \xi(t)$ is continuous from $\Cl0{}([0,T];\Lp2{}(0,L))$ into $\Lp2{}(0,L)$, from \eqref{181} we deduce
that
$$
\chi_k(\cdot,t)\wto\chi(\cdot,t)\qquad\text{weakly in $\Lp2{}(0,L)$}
$$
for every $t\in[0,T]$. Then \eqref{1804} gives that
\begin{subequations}\label{18**}
\begin{eqnarray}
&\label{1810}
\chi(\cdot,t)\in \Hp2{}(0,L)\quad\text{for every $t\in[0,T]$,}
\\
&\label{1805}
\chi_k(\cdot,t)\wto\chi(\cdot,t)\quad\text{weakly in $\Hp2{}(0,L)$ for every $t\in[0,T]$,}
\\
&\label{1806}\displaystyle
\sup_{t\in[0,T]} \norma{\chi(\cdot,t)}_{\Hp2{}(0,L)}\leq C.
\end{eqnarray}
\end{subequations}

By \eqref{1811} and \eqref{1806} we have $\chi\in X$. Since the embedding of $\Hp2{}(0,L)$ into $\Cl1{}([0,L])$ is compact, from \eqref{1805} we deduce that $\chi_k(\cdot,t)\to\chi(\cdot,t)$ strongly in $\Cl1{}([0,L])$ for every $t\in[0,T]$. This allows us to pass to the limit in the equalities $\norm{\chi_k^\prime(s,t)}=1$, $\chi_k(s,0)=\chi_\init(s)$, $\chi_k(s,T)=\chi_\fin(s)$, and in the external disks condition with radius $\rho$. We conclude that $\chi\in X_1^\rho$ and that $\chi(s,0)=\chi_\init(s)$, and $\chi(s,T)=\chi_\fin(s)$.

\par Let us verify that also the force and torque balance \eqref{105} passes to the limit. Equality \eqref{105a} for $\chi_k$ reads
\begin{equation}\label{1809}
\int_0^L K_{\chi_k}(s,t)\dot\chi_k(s,t)\,\de s=0.
\end{equation}
Since $\chi_k'(\cdot,t)$ converges  to $\chi'(\cdot,t)$ strongly in $\Cl1{}([0,L])$ for every $t\in[0,T]$, by \eqref{0000} and  \eqref{1806} we can apply the Dominated Convergence Theorem and we obtain 
\begin{equation}\label{1807}
K_{\chi_k}\to K_\chi \quad\text{strongly in $\Lp2{}(0,T;\Lp2{}(0,L))$.}
\end{equation}
By \eqref{181} we have also
\begin{equation}\label{1808}
\dot\chi_k\wto\dot\chi\quad\text{weakly in $\Lp2{}(0,T;\Lp2{}(0,L))$.}
\end{equation}
By \eqref{200} and by the Ioffe-Olech semicontinuity theorem (see, for instance, \cite[Theorem 2.3.1]{Butt})
we have
\begin{equation}\label{1813}
\cP(\chi)\leq\liminf_{k\to+\infty}\cP(\chi_k).
\end{equation}

Let now $\varphi\in \Cl0c(0,T)$ be a test function. By \eqref{1809}-\eqref{1808} we have
$$0=\int_0^T\!\!\int_0^L \varphi(t)K_{\chi_k}(s,t)\dot\chi_k(s,t)\,\de s\de t\to
\int_0^T\!\!\int_0^L \varphi(t)K_{\chi}(s,t)\dot\chi(s,t)\,\de s\de t
$$
Since this equality holds for every $\varphi\in \Cl0c(0,T)$, we conclude that
$$\int_0^L K_{\chi}(s,t)\dot\chi(s,t)\,\de s=0$$
for almost every $t\in[0,T]$. This proves 
\eqref{105a}. The proof of \eqref{105b} is analogous.
 Since $(\chi_k)_k$ is a minimizing sequence, we deduce from \eqref{1813} that $\chi$ is a minimizer of 
 \eqref{180}.
\end{proof}

\bigskip

\noindent {\bf Acknowledgments.} This material is based on work supported by the Italian 
Ministry of Education, University, and
Research under the Projects
PRIN 2008
``Variational Problems with
Multiple Scales" and
PRIN 2010-11 ``Calculus of Variations"
and by the European Research Council through the ERC Advanced Grant 340685\_MicroMotility.
The work of M.M.\@ was partially supported by grant FCT-UTA\_CMU/MAT/0005/2009 ``Thin Structures, Homogenization, and Multiphase Problems''.

\end{document}